\documentclass[11pt]{article} 

\usepackage{amsmath, amscd, amsthm, amssymb, mathrsfs} 
\usepackage[applemac]{inputenc}
\usepackage[all]{xy} 

\usepackage{vmargin}
\setpapersize{A4}
\setmarginsrb{3cm}{3cm}{3cm}{3cm}{0cm}{0cm}{1cm}{1cm}
\DeclareMathOperator{\re}{Re}

\DeclareMathOperator{\Orth}{O}
\DeclareMathOperator{\SO}{SO}

\DeclareMathOperator{\U}{U}
\DeclareMathOperator{\SU}{SU}

\DeclareMathOperator{\SL}{SL}
\DeclareMathOperator{\PSL}{PSL}
\DeclareMathOperator{\PSU}{PSU}

\DeclareMathOperator{\Isom}{Isom}

\newcommand{\R}{\mathbb R}
\newcommand{\C}{\mathbb C}
\newcommand{\Z}{\mathbb Z}
\newcommand{\N}{\mathbb N}

\newcommand{\diff}{\text{\rm d}}
\newcommand{\del}{\partial}

\newcommand{\so}{\mathfrak{so}}

\newcommand{\g}{\mathfrak{g}}

\renewcommand{\P}{\mathbb P}
\renewcommand{\H}{\mathbb H}

\renewcommand{\u}{\mathfrak{u}}

\theoremstyle{plain}
	\newtheorem{theorem}{Theorem}
	\newtheorem{proposition}[theorem]{Proposition}
	\newtheorem{lemma}[theorem]{Lemma}
	\newtheorem{corollary}[theorem]{Corollary}

	\newtheorem{question}[theorem]{Question}
\theoremstyle{definition}

\theoremstyle{plain}
	\newtheorem*{theorem*}{Theorem}
	\newtheorem*{proposition*}{Proposition}
	\newtheorem*{lemma*}{Lemma}
	\newtheorem*{corollary*}{Corollary}
	\newtheorem*{conjecture*}{Conjecture}
\theoremstyle{definition}
	\newtheorem*{definition*}{Definition}
	\newtheorem*{remark*}{Remark}
	\newtheorem*{remarks*}{Remarks}
	
\makeatletter
\def\blfootnote{\xdef\@thefnmark{}\@footnotetext}
\makeatother

\begin{document}

\title{Hyperbolic geometry and non-Kähler manifolds\\ with trivial canonical bundle}
\author{Joel Fine and Dmitri Panov}
\date{}

\maketitle

%\blfootnote{First author supported by an FNRS chargé de recherche fellowship.}
%\blfootnote{Second author supported by EPSRC grant EP/E044859/1.}

\abstract{
We use hyperbolic geometry to construct simply-connected symplectic or complex manifolds with trivial canonical bundle and with no compatible Kähler structure. We start with the desingularisations of the quadric cone in $\C^4$: the smoothing is a natural $S^3$-bundle over $H^3$, its \emph{holomorphic} geometry is determined by the hyperbolic metric; the small-resolution is a natural $S^2$-bundle over $H^4$ with \emph{symplectic} geometry determined by the metric. Using hyperbolic geometry, we find orbifold quotients with trivial canonical bundle; smooth examples are produced via crepant resolutions. In particular, we find the first example of a simply-connected symplectic 6-manifold with $c_1=0$ that does not admit a compatible Kähler structure. We also find infinitely many distinct complex  structures on $2(S^3\times S^3)\#(S^2\times S^4)$ with trivial canonical bundle. Finally, we explain how an analogous construction for hyperbolic manifolds in higher dimensions gives symplectic non-Kähler ``Fano'' manifolds of dimension 12 and higher.
}

\setcounter{tocdepth}{2}
\tableofcontents

\section{Introduction}

The main goal of this article is to describe how hyperbolic geometry can be used to construct simply-connected complex and symplectic manifolds with \emph{trivial canonical bundle}, which admit \emph{no compatible Kähler structure}. In the symplectic case, trivial means $c_1=0$, whilst in the complex case, trivial means holomorphically trivial. Our examples have real dimension six. Hyperbolic geometry in dimension three gives complex examples whilst hyperbolic geometry in dimension four leads to symplectic examples. 

Using this approach we find a symplectic structure on a certain 6-manifold. This 6-manifold is simply connected, has vanishing first Chern class and admits no compatible Kähler structure, giving the first such 6-dimensional example, at least to the best of our knowledge. On the complex side we find, amongst other things, infinitely many distinct complex structures on $2(S^3 \times S^3)\#(S^2 \times S^4)$ all with trivial canonical bundle. 

Before giving an outline of the constructions, we begin by recalling some background.

\subsection{Non-Kähler examples in dimension four}

First consider the case of real dimension four. It follows from the classification of compact complex surfaces that the only simply-connected compact complex surfaces with trivial canonical bundle are K3 surfaces and so, in particular, are Kähler. In the symplectic case the situation is not yet understood, but it is a folklore conjecture that Kähler K3 surfaces again provide the only examples. So, conjecturally,  in dimension four all simply-connected symplectic manifolds with $c_1=0$ are Kähler too. If we believe this conjecture, to find the first non-Kähler examples of simply-connected manifolds with trivial canonical bundle, we should look to higher dimensions.

\subsection{Non-Kähler examples in dimension six}

In dimension six, much use has been made of ordinary double-points and their desingularisations. An ordinary double-point is modelled locally on a neighbourhood of the origin in the quadric cone $\sum z_j^2=0$ in $\C^4$. To desingularise the cone, one can smooth it to obtain the smooth affine quadric $\sum z_j^2 = 1$; this replaces the double point by a three-sphere. Alternatively, one can resolve the double-point by the so-called small resolution, which is the total space of $\mathcal O(-1) \oplus \mathcal O(-1)  \to \C\P^1$; here the double point is replaced by a two-sphere. (We will give more detail about these desingularisations later in \S\ref{hyperbolic geometry and the conifold}.) Note that the double point and both desingularisations have trivial canonical bundle. 

In the complex setting this approach was used by Friedman \cite{friedman} and Tian \cite{tian}, following ideas of Clemens \cite{clemens}, to produce complex manifolds with trivial canonical bundle. Starting from a complex manifold $X$ with trivial canonical bundle and ordinary double points they considered the manifold $X'$ obtained by smoothing all the double points. Under certain conditions, the complex structure on $X$ can be smoothed to give a complex structure on $X'$, again with trivial canonical bundle. In this way one finds a family of complex structures on $n(S^3 \times S^3)$ for all $n\geq 2$,  with trivial canonical bundle. (The article of Lu--Tian \cite{lu-tian} explains the details for these particular smoothings.) These manifolds are clearly not Kähler.

There is an analogous story on the symplectic side, explained in the article of Smith--Thomas--Yau \cite{smith-thomas-yau}. Here, one starts with a symplectic manifold $Y$ with ordinary double points and trivial canonical bundle and considers the manifold $Y'$ obtained by the small resolution of all the double points. Under certain conditions, the symplectic structure on $Y$ can be resolved to give a symplectic structure on $Y'$ with $c_1=0$. In this way Smith--Thomas--Yau produced many interesting examples of simply-connected symplectic manifolds with $c_1=0$ which are widely believed to be non-Kähler. However, in contrast with the complex side, they were unable to find examples which violate the standard Kähler topological restrictions. It is still unknown if any of the examples they give actually are non-Kähler.

\subsection{Non-Kähler examples in dimension $4n$}

In \cite{guan}, Guan constructs examples of simply-connected compact complex-symplectic manifolds of dimension $4n$ for $n \geq 2$, which do not admit a compatible Kähler structure. (See also the article \cite{bogomolov} of Bogomolov for an alternative exposition of Guan's construction.) Taking the real part of the complex symplectic form, one obtains a symplectic manifold in the real sense which has $c_1=0$. In \cite{guan}, Guan proves that his examples are not diffeomorphic to hyperkähler manifolds. In fact even more is true; whilst it is not proven directly in \cite{guan}, Guan has informed us that it follows from the results of \cite{guan} that his manifolds admit no Kähler structure whatsoever. 

Guan's construction uses Kodaira--Thurston surfaces to produce a single example of a simply-connected $4n$-manifold for each $n \geq 2$ which admits complex-symplectic structures. One might think of Guan's construction as a non-Kähler analogue of the infinite series of hyperkähler manifolds coming from the Hilbert scheme of $n$ points on a hyperkähler surface. In contrast, the construction we present here produces large numbers of topologically distinct symplectic manifolds all of the same dimension. Whilst we describe just one in this article, we produce examples with arbitrarily high Betti numbers in a sequel \cite{fine-panov3}. The idea is outlined here in \S\ref{concluding remarks}.

\subsection{A hyperbolic picture of the desingularisations}

As in the articles \cite{clemens,friedman,tian,smith-thomas-yau} mentioned above, the construction of symplectic and complex manifolds that we explore in this article relies on the desingularisations of the threefold quadric cone. However, rather than using the small resolution or smoothing as local models for desingularising double-points, here we consider their \emph{global} geometry. 

\subsubsection{$H^4$ and the small resolution}

As we explain in \S\ref{hyperbolic geometry and the small resolution}, the symplectic geometry of the small resolution $R = \mathcal O(-1)\oplus \mathcal O(-1)$ is intimately related to hyperbolic geometry in dimension four. $R$ is a natural $S^2$-bundle over $H^4$, with the symplectic Calabi--Yau structure on $R$ determined by the metric on $H^4$. Indeed, $R$ is symplectomorphic to the twistor space of $H^4$; the symplectic structure on the twistor space was first defined by Reznikov \cite{reznikov} and  Davidov--Mu\v{s}karov--Grantcharov \cite{davidov} and a symplectomorphism with $R$ was given in \cite{fine-panov}. In \S\ref{quaternionic description} and \S\ref{coadjoint description} we give two additional descriptions: symplectically, $R$ can be seen as a coadjoint orbit of $\SO(4,1)$, or as a certain pseudo-Kähler manifold, via a construction involving quaternions, analogous to the twistor fibration $\C\P^3 \to \H\P^1\cong S^4$. It follows that any hyperbolic 4-manifold carries an $S^2$-bundle---its twistor space---whose total space is a symplectic manifold with trivial canonical bundle. 

To produce a simply connected example, in \S\ref{symplectic example} we use a version of the Kummer construction. Recall that in the Kummer construction, one considers the quotient of an abelian surface by the central involution $z\mapsto -z$. Resolving the 16 orbifold points gives a K3 surface. 

The role of the abelian surface in our situation is played by a beautiful hyperbolic 4-manifold called the Davis manifold \cite{davis} (see also the description in \cite{ratcliffe-tschantz}). This manifold is constructed by gluing opposite 3-faces of a 120-cell in $H^4$, a certain regular four-dimensional polyhedron (so-called because it has 120 three-dimensional faces, each a regular dodecahedron). 

An important point for our purposes is that reflection in the centre of the 120-cell gives an isometry of the Davis manifold $M$. This isometry gives a quotient $M/\Z_2$ which is a simply-connected singular manifold with 122 isolated singularities modelled on $z\mapsto -z$, just as in the Kummer construction. In terms of the symplectic $S^2$-bundle $Z \to M$, the isometry induces a symplectic action of $\Z_2$ on $Z$ which fixes 122 fibres. The quotient is a symplectic orbifold with trivial canonical bundle, with singularities along 122 $S^2$'s. Making a symplectic crepant resolution of these singularities gives a smooth simply-connected symplectic manifold $\hat Z$. As we will see, $b_3(\hat Z)=0$ which can never happen for a simply-connected Kähler manifold with $c_1=0$. 

\subsubsection{$H^3$ and the smoothing}

There is an analogous story for the smoothing $S = \{\sum z_j^2=1\}$, this time relating the complex geometry of $S$ with hyperbolic geometry in dimension three. $S$ is a natural $S^3$-bundle over $H^3$ with the complex geometry of $S$ determined by the metric on $H^3$. One way to see this is to note that $S \cong \SL(2,\C)$; the action of $\SU(2)$ on $\SL(2,\C)$ makes it a principal bundle over $\SL(2,\C)/\SU(2) \cong H^3$, the principal spin bundle of $H^3$. As a consequence, any hyperbolic 3-manifold carries an $S^3$-bundle whose total space is a complex manifold with trivial canonical bundle.  

To produce simply-connected examples, in \S\ref{complex examples} we consider hyperbolic orbifold metrics on $S^3$ with cone angle $2\pi/m$ along a knot $K$. It is a standard fact that there is a smooth hyperbolic manifold $M$ admitting an isometric $\Z_m$-action with quotient $M\to S^3$ branched over $K$. Now, as $M$ is hyperbolic, its spin-bundles are complex manifolds with trivial canonical bundle. As we will show, provided the spin-structure  $P \to M$ is well-chosen, the generator of the $\Z_m$-action lifts to $P$ where it generates a $\Z_{2m}$-action (covering the $\Z_m$-action on $M$ via the projection $\Z_{2m} \to \Z_m$). 

The quotient $P/\Z_{2m}$ is just the quotient of the frame bundle $Q \to M$ by the lift of the action of $\Z_m$. It is straightforward to see that $\Z_m$ acts freely on $Q$ and so $Q/\Z_m$ is smooth and has fundamental group equal to the original orbifold fundamental group of the metric on $S^3$. To remedy this, we ``twist'' the $\Z_{2m}$-action on $P$ around the knot $K$; this induces fixed points in the fibres over $K$ and results in a singular but simply-connected quotient of $P$. Finally, we obtain a smooth manifold by taking a crepant resolution.

The examples we find this way have a smooth surjection $X \to S^3$. Away from the knot, this makes $X$ a trivial fibration $S^3 \times (S^3\setminus K)$. These threefolds can never be Kähler; as we explain, they admit a  $\C^*$-action with no fixed points. This construction provides, amongst other things, infinitely many distinct complex structures with trivial canonical bundle on $2(S^3\times S^3)\#(S^2\times S^4)$.

\subsection{Symplectic non-Kähler ``Fano'' manifolds}

In \S\ref{symplectic fanos} we outline a relationship between symplectic and hyperbolic geometry in all even dimensions: a hyperbolic manifold of dimension $2n$ is the base of a fibre bundle whose total space is a symplectic manifold of real dimension $n(n+1)$. (The fibre bundle is the twistor space; these symplectic manifolds were first considered by Reznikov \cite{reznikov}.) When $n=1$ the fibres are zero-dimensional and the total space is just the original surface. In this case, the symplectic manifold is ``general type'' with symplectic class a negative multiple of $c_1$. When $n=2$, we have the situation described above, for which $c_1=0$. For all higher dimensions, however, it turns out that the symplectic manifold has symplectic class equal to a positive multiple of $c_1$. No compact example produced this way can be Kähler (e.g., for fundamental group reasons). To the best of our knowledge, these examples (found first by Reznikov) are the first symplectic non-Kähler manifolds which have symplectic class equal to a positive multiple of $c_1$, symplectic analogues of Fano varieties.

It is interesting to compare this with the situation in dimension four. There, it is a consequence of the work of Gromov \cite{gromov}, Taubes \cite{taubes} and McDuff \cite{mcduff} that a symplectic Fano 4-manifold must be Kähler. Meanwhile, the lowest dimension attained by our construction is twelve. It is natural to ask for the minimal dimension in which non-Kähler examples occur. Do they exist in dimension six? The symplectic Fanos coming from hyperbolic $2n$-manifolds also have the property that $c_1^n$ can be arbitrarily large (it is essentially the volume of the hyperbolic $2n$-manifold). Again, it seems natural to ask for the minimal dimension in which Fano manifolds exist for which $c_1^n$ can be made arbitrarily large.

\subsection{Acknowledgements}

It is a pleasure to thank the following people for discussions held during the course of this work: Mohammed Abouzaid, Michel Boileau, Alessio Corti, Simon Donaldson, Daniel Guan, Maxim Kontsevich, Federica Pasquotto, Vicente Mu\~noz, Brendan Owens, Simon Salamon, Paul Seidel, Ivan Smith, Burt Totaro, Richard Thomas, Henry Wilton, Claire Voisin, Jean-Yves Welschinger and Dominic Wright. We would also like to thank an anonymous reader who pointed out an error in an earlier version of this work. 

The first author was supported by an FNRS chargé de recherche fellowship. The 
second author was supported by EPSRC grant EP/E044859/1.

\section{Hyperbolic geometry and the conifold}
\label{hyperbolic geometry and the conifold}

\subsection{Desingularisations of the conifold}

Fix a non-degenerate complex quadratic form $q$ on $\C^4$. The \emph{conifold} is the affine quadric cone $Q = \{q(\zeta) =0 \}$. We will consider two well-studied ways to remove the singularity at the origin, giving smooth Kähler manifolds with trivial canonical bundle. We describe this briefly here; for more details we refer to \cite{smith-thomas-yau}, from where we originally learnt this material. 

One desingularisation is the smooth affine quadric $S = \{q(\zeta) = 1\}$, called the \emph{smoothing} of $Q$. $S$ is a Kähler manifold with trivial canonical bundle, although we will be concerned only with the \emph{complex} geometry of $S$. 

Another way to desingularise the conifold is to take a resolution. Choose coordinates $(x,y,w,z)$ so that the quadratic form is given by $q(x,y,z,w,) = xw-yz$. Let $R$ denote the total space of $\pi \colon \mathcal O(-1) \oplus \mathcal O(-1) \to \C\P^1$. Each summand of $R$ has a natural map $\mathcal O(-1) \to \C^2$; these combine to give a map $p \colon R \to \C^2 \oplus \C^2 \cong \C^4$.  Here, we use the identification $((a,b),(c,d)) \mapsto (a,b,c,d)$. The image of $p$ consists of points $(x,y,z,w)$ such that $[x:y] = [z:w]$, i.e., such that $xw-yz=0$, which is the conifold. So  $p \colon R \to Q$ is a resolution, called the \emph{small resolution}, in which the singularity has been replaced by the zero section in~$R$. 

In this description, we could equally have chosen to identify $\C^2 \oplus \C^2 \cong \C^4$ by $((a,b),(c,d)) \mapsto (a,c,b,d)$. With this choice, the image of $p \colon R \to \C^4$ consists of points $(x,y,z,w)$ such that $[x:z] = [y:w]$ which is again the conifold $xw-yz=0$. So, in fact, there are two inequivalent ways to view $R$ as the small resolution of $Q$. Put another way, there are \emph{two} small resolutions $R_{\pm}\to Q$; each is isomorphic to $\mathcal O(-1)\oplus \mathcal O(-1)$ abstractly, but there is no isomorphism between them which respects the projections to $Q$.

There is an alternative coordinate-free description of the small resolutions. Blow up the origin in $\C^4$ to obtain $\widetilde \C^4$. The proper transform $\widetilde Q$ of $Q$ meets the exceptional $\C\P^3 \subset \widetilde \C^4$ in a quadric surface. This surface is biholomorphic to $\C\P^1 \times \C\P^1$ and each of the rulings has negative normal bundle. Blowing down one or other of the rulings gives the two small resolutions $R_{\pm} \to Q$ of the conifold.

Either resolution $R$ is a Kähler manifold with trivial canonical bundle, although we will be concerned only with the \emph{symplectic} geometry of $R$. The symplectic form is given by 
$$
\omega_R = \pi^* \omega_{\C\P^1} + p^* \omega_{\C^4}
$$
where $\pi \colon R \to \C\P^1$ and $p \colon R \to Q \subset\C^4$ are the vector-bundle and resolution projections respectively. 

The first hint of a link with twistor geometry is provided by considering the symplectic action of $\SO(4)$ on $R$. The Hermitian metric and complex quadratic form on $\C^4$ define a choice of conjugation map (the real points are those on which the Hermitian and complex forms agree). Then $\SO(4)$ is the subgroup of $\U(4)$ which commutes with this conjugation. In this way $\SO(4)$ acts by Kähler isometries on the conifold $Q$ and this action lifts to a Kähler action on the small resolutions $R_\pm$. The action on the exceptional $\C\P^1$ in $R_\pm$ is given by one or other of the projections $\SO(4) \to \PSU(2)$ arising from the exceptional isomorphism
\begin{equation}\label{real exceptional isomorphism}
\SO(4) \cong \frac{\SU(2)\times \SU(2)}{\pm1}.
\end{equation}

\subsection{Hyperbolic geometry and the smoothing}
\label{S and H^3}

To describe the connection between hyperbolic geometry and $S$ the smoothing, use coordinates $(x, y, z, w)$ in which the quadratic form $q$ is given by $q(x,y,z,w) = xw-yz$. Identify $\C^4$ with the set of complex $2\times 2$ matrices by
$$
\left(
	\begin{array}{cc}
		x & y\\
		z & w
	\end{array}
\right)
\mapsto
(x,y,z,w).
$$
When evaluated on a matrix $A$, the quadratic form is $q(A) = \det A$. Hence, in this picture, $S \cong \SL(2, \C)$ is given by matrices with determinant 1. Consider the action of $\SU(2)$ on $\SL(2,\C)$ given by multiplication on the left; this makes $\SL(2,\C)$ into a principal $\SU(2)$-bundle over the symmetric space $\SL(2,\C)/\SU(2)$ which is precisely hyperbolic space $H^3$. The bundle $\SL(2,\C) \to H^3$ is the principal spin bundle of $H^3$. To verify this note that $\Isom(H^3) \cong \PSL(2,\C)$ acts freely and transitively on the frame bundle of $H^3$; so $\PSL(2,\C)$ can be identified with the frame bundle of $H^3$  whilst its double cover $\SL(2,\C)$ is identified with the principal spin bundle. 

As it is a complex Lie group,  $\SL(2,\C)$ has a holomorphic volume form which is invariant under right-multiplication. It will be essential to us that this same form is also invariant under \emph{left}-multiplication; in other words, that $\SL(2,\C)$ is a complex unimodular group. Given $P,Q \in \SL(2,\C)$, the map $A \mapsto P^{-1}AQ$ sets up an isomorphism 
\begin{equation}\label{exceptional isomorphism}
\SO(4,\C) \cong \frac{\SL(2,\C) \times \SL(2,\C)}{\pm 1}.
\end{equation}
(This is, of course, just the complexification of (\ref{real exceptional isomorphism}).) Unimodularity of $\SL(2,\C)$ now amounts to the following:

\begin{proposition}\label{action on S}
$\SO(4,\C)$ acts by biholomorphisms on $S$. Moreover, $S$ admits an invariant holomorphic volume form.
\end{proposition}

Alternatively, the $\SO(4,\C)$-invariant volume form can be seen directly, without reference to $\SL(2, \C)$ and unimodularity. $\SO(4, \C)$ preserves both the Euclidean holomorphic volume form $\Omega_{0}$ and the radial vector $e = x\del_x + y\del_y + z \del_z + w\del_w$, which is transverse to $S$. Hence the volume form $\Omega = \iota_e \Omega_0$ on $S$ is also invariant.

To produce compact quotients of $S$, let $M$ be an oriented compact hyperbolic 3-manifold with fundamental group $\Gamma \subset \PSL(2,\C)$. Since $M$ is spin (as all oriented 3-manifolds are), $\Gamma$ lifts to $\SL(2,\C)$ and the total space of the principal spin bundle of $M$ is $X = \SL(2,\C)/\Gamma$, where $\Gamma$ acts by right-multiplication. More precisely, each choice of lift of $\Gamma$ to $\SL(2,\C)$ gives a spin structure on $M$;  this correspondence is one-to-one: lifts and spin structures are both parametrised by $H^1(M, \Z_2)$. Whatever the choice of lift, the action of $\Gamma$ preserves the holomorphic volume form on $\SL(2,\C)$ and so $X$ is a compact complex threefold with trivial canonical bundle. 

Left-multiplication gives us additional freedom. Denote the chosen lift by $\alpha \colon \Gamma \to \SL(2,\C)$ and let $\rho \colon \Gamma \to \SL(2, \C)$ be some other homomorphism. From $\rho$ we obtain a new action of $\Gamma$ on $S$:
$$
\gamma \cdot A = \rho(\gamma)^{-1} A\, \alpha(\gamma).
$$
We will use this ``twisting'' to produce orbifold quotients of $S$, but it was originally used in a different context by Ghys \cite{ghys}. Ghys studies infinitesimal deformations of the complex manifold $X = \SL(2,\C)/\Gamma$, where $\Gamma$ acts by right-multiplication, i.e., where $\rho$ is the trivial homomorphism in the above picture. He shows that infinitesimal holomorphic deformations of $X$ are equivalent to infinitesimal deformations of the trivial homomorphism. It is possible to extend this picture, giving a description of part of the space of complex structures on $X$, something we address in forthcoming work \cite{fine-panov2}. 

\subsection{Hyperbolic geometry and the small resolution}
\label{hyperbolic geometry and the small resolution}

This section proves the analogue of Proposition \ref{action on S} for the small reso\-lution~$R$. 

\begin{proposition}\label{action on R}
$\SO(4,1)$ acts symplectomorphically on $R$, extending the action of $\SO(4)$. Moreover, $R$ admits an invariant compatible almost complex structure and invariant complex volume form.
\end{proposition}

Before proving this, we give three different descriptions of the hyperbolic geometry of $R$.

\subsubsection{Twistorial description}

The twistor space of $H^4$ carries a natural symplectic structure (a fact noticed independently by Reznikov \cite{reznikov} and Davidov--Mu\v{s}\-karov--Grantcharov \cite{davidov}) and this was shown to be symplectomorphic to the small resolution in \cite{fine-panov}. We very briefly recall the idea here. 

Choose coordinates $z_j$ on $\C^4$ in which the conifold is $\{\sum z_j^2=0\}$. The map $Q \to \R^4$ given by $z \mapsto \re z$ exhibits $Q\setminus 0$ as an $S^2$-bundle over $\R^4 \setminus 0$. The fibre over a point $x$ is all points of the form $x+iy$ where $y\in \langle x \rangle^\perp$ with $|y| = |x|$.

The \emph{twistor space} of $\R^4$ is the bundle of unit-length self-dual two forms. Given $x \in\R^4 \setminus 0$, interior contraction with $x$ gives an isomorphism $\Lambda^+ \cong \langle x \rangle^\perp$. In this way we can identify $Q \setminus 0$ with the twistor space of $\R^4\setminus 0$. This identification extends over the small resolution to give an $\SO(4)$-equivariant identification of $R$ with the twistor space of $\R^4$. (More precisely it extends over one of the two small resolutions, to obtain the other small resolution we should consider $\Lambda^-$.)

It was proved in \cite{fine-panov}  that, up to homotheties and rescaling, there is a unique $\SO(4)$-invariant symplectic form on the twistor space $R$ of $\R^4$ with infinite volume and whose sign changes under the antipodal map. From this twistorial view-point, it comes naturally from hyperbolic geometry. The Levi--Civita connection of the hyperbolic metric on $\R^4 \cong H^4$ induces a metric connection in the vertical tangent bundle $V \to R$. Its curvature $-2\pi i \omega$ determines the symplectic form. It follows from this that there is a symplectic action of the hyperbolic isometry group $\SO(4,1)$ extending that of $\SO(4)$.

In \cite{davidov2}, Davidov--Mu\v{s}karov--Grantcharov compute the total Chern form of the tangent bundle of the twistor space. Combining their calculation with the symplectomorphism of the twistor space and the small resolution in \cite{fine-panov} gives a proof of Proposition \ref{action on R}. We give a separate proof later, using an alternative description of $R$. 

\subsubsection{Quaternionic description}
\label{quaternionic description}

The next description, involving quaternions, is analogous to the standard picture of the twistor fibration $\C\P^3 \to S^4$ which we recall first. Begin by identifying $\C^4 \cong \H^2$. Each complex line in $\C^4$ determines a quaternionic line in $\H^2$, giving a map $t \colon \C\P^3 \to \H\P^1\cong S^4$. The fibre $t^{-1}(p)$ is the $\C\P^1$ of all complex lines in the quaternionic line~$p$. A choice of positive-definite Hermitian form on $\C^4$ gives a Fubini--Study metric on $\C\P^3$ and a round metric on $S^4$. The isometries $\SO(5)$ of $S^4$ are identified with those isometries of $\C\P^3$ which preserve the fibres of $t$, giving an injection $\SO(5) \to \PSU(4)$.

To describe the twistor fibration of $H^4$ in an analogous way we use an \emph{indefinite} Hermitian form: let $\C^{2,2}$ denote $\C^4$ together with the Hermitian form $h(w) = |w_1|^2 + |w_2|^2 - |w_3|^2 - |w_4|^2$ and consider the space
$$
N = \{ w\in \C^{2,2} : h(w)<0\}/\C^*
$$ 
of negative lines in $\C^{2,2}$; $N$ is an open set in $\C\P^3$. Transverse to a negative line in $\C^{2,2}$, $h$ is indefinite with signature $(2,1)$; accordingly $N$ inherits a \emph{pseudo}-Kähler metric of signature $(2,1)$, in the same way that a positive definite Hermitian form determines a Fubini--Study metric on $\C\P^3$. The pseudo-Kähler metric makes $N$ into a symplectic manifold. Alternatively, one can see the $N$ as the symplectic reduction of $\C^{2,2}$ by the diagonal circle action; the Hamiltonian for this action is just $h$ and the reduction at $h=-1$ is $N$. 

It is standard that $N$ is symplectomorphic to $R$. (One way to prove this is to use the fact that both admit Hamiltonian $T^3$-actions with equivalent moment polytopes.) To use this picture to relate $R$ to $H^4$, let $\H^{1,1}$ denote $\H^2$ together with the indefinite form $h'(p) = |p_1|^2 - |p_2|^2$. The map
$$
(w_1,w_2,w_3,w_4) \mapsto (w_1 + jw_2, w_3 + jw_4)
$$
identifies the Hermitian spaces $\C^{2,2} \cong \H^{1,1}$. There is a natural projection from $N$ to the space of negative quaternionic lines in $\H^{1,1}$, i.e., to the quaternionic-hyperbolic space:
$$
H^1_\H = \{ p \in \H^{1,1} : h'(p) <0\} /\H^*.
$$
In general, the space of negative lines in $\H^{n,1}$ is the quaternionic analogue of hyperbolic or complex-hyperbolic $n$-space. When $n=1$, however, this is isometric to $H^4$, four-dimensional real-hyperbolic space. (In the lowest dimension, the symmetric spaces associated to complex or quaternionic geometry coincide with their equidimensional real analogues.) 

So, analogous to $t \colon \C\P^3 \to S^4$, we have a projection $t \colon R \to H^4$; the fibre $t^{-1}(p)$ is the $\C\P^1$ of all complex lines in the quaternionic line~$p$. The pseudo-Kähler isometries of $R$ are $\PSU(2,2)$ whilst the isometries $\SO(4,1)$ of $H^4$ can be identified with those isometries of $R$ which preserve the fibres of $t$, giving an injection $\SO(4,1) \to \PSU(2,2)$. In this way we see again an action of $\SO(4,1)$ on $R$ by symplectomorphisms. 

\subsubsection{A coadjoint description}
\label{coadjoint description}

Let $G$ be a Lie group and $\xi \in \g^*$; denote the orbit of $\xi$ under the coadjoint action by $\mathcal O(\xi)$. It is a standard fact that there is a $G$-invariant symplectic structure on $\mathcal O(\xi)$. We will show how the small resolution fits into this general theory as a certain coadjoint orbit of $\SO(4,1)$. 

The Lie algebra $\so(4,1)$ is  5$\times$5 matrices of the form
\begin{equation}
\label{matrix}
\left(
\begin{array}{cc}
0 & u^t\\
u & A
\end{array}
\right),
\end{equation}
where $u$ is a column vector in $\R^4$ and $A \in \so(4)$. Those elements with $u=0$ generate $\SO(4)\subset \SO(4,1)$. 

The Killing form is non-degenerate on $\so(4,1)$ and so gives an equivariant isomorphism $\so(4,1) \cong \so(4,1)^*$. We consider the orbit of
$$
\xi=\left(
\begin{array}{cc}
0 & 0\\
0 & J_0
\end{array}
\right)
$$
where $J_0\in \so(4)$ is a choice of almost complex structure on $\R^4$ (i.e., $J_0^2=-1$). The subalgebra $\mathfrak h$ of matrices commuting with $\xi$ is those with $u=0$ and $[A,J_0]=0$, i.e., $\mathfrak h = \u(2) \subset \so(4) \subset\so(4,1)$. It follows that the stabiliser of $\xi$ is $\U(2)$ and so $\mathcal O(\xi) \cong \SO(4,1)/\U(2)$.

\begin{lemma}\label{U(2) decomposition}
There is an isomorphism of $\U(2)$-representation spaces:
$$
\so(4,1) \cong \u(2)\oplus \Lambda^2(\C^2)^* \oplus \C^2 .
$$
\end{lemma}

\begin{proof}
There is a $\U(2)$-equivariant isomorphism $\so(4) \cong \u(2) \oplus \Lambda^2(\C^2)^*$. To see this, write $\so(4) \cong \Lambda^2(\R^4)^*$. Given a choice of almost complex structure on $\R^4$, any real 2-form $a$ can be written uniquely as $a= \alpha+\beta+\bar\beta$ where $\alpha \in \Lambda^{1,1}_\R$ is a real $(1,1)$-form and $\beta \in \Lambda^{2,0}$. Identifying $a$ with $(\alpha, \beta)$ gives a $\U(2)$-equivariant decomposition $\Lambda^2_\R \cong \Lambda^{1,1}_\R\oplus \Lambda^{2,0}$. But, via the Hermitian form, $\Lambda^{1,1}_\R$ is identified with skew-Hermitian matrices $\u(2)$ and this gives the claimed isomorphism.

There is also an $\SO(4)$-equivariant isomorphism $\so(4,1) \cong \so(4) \oplus \R^4$. In the form (\ref{matrix}), the $\so(4)$ summand is given by $u=0$ whilst the $\R^4$ summand by $A=0$. Combining these two isomorphisms completes the proof.
\end{proof}

\begin{lemma}\label{uniqueness of structure}
Up to scale, there is a unique $\SO(4,1)$-invariant symplectic form on $\SO(4,1)/\U(2)$.
\end{lemma}

\begin{proof}
The existence follows from coadjoint orbit description. For uniqueness, we begin by describing all invariant non-degenerate 2-forms. This amounts to describing all non-degenerate 2-forms at a point which are invariant under the stabiliser $\U(2)$. From Lemma \ref{U(2) decomposition} the tangent space at a point is isomorphic as a $\U(2)$-representation space to $\Lambda^2(\C^2)^* \oplus \C^2$. Up to scale, there is a unique invariant 2-form on each summand, giving two $\SO(4,1)$-invariant 2-forms $a, b$ on $\SO(4,1)/\U(2)$. It follows that, up to scale, all non-degenerate invariant 2-forms have the form $a+tb$ for $t \neq 0$. At most one of these can be closed, $t$ being fixed by the requirement that $\diff a = t\diff b$.
\end{proof}

\begin{corollary}
$\mathcal O(\xi)$ is symplectomorphic to $R$.
\end{corollary}

\begin{proof}
This follows from the previous lemma and the transitive action of $\SO(4,1)$ on $R$ which can be seen in either the twistorial or quaternionic pictures.
\end{proof}

We make a small digression to point out a simple consequence of the coadjoint orbit description of $R$ which is, perhaps, less apparent from its Kähler description as the total space of $\mathcal O(-1)\oplus \mathcal O(-1)$. We omit the details since this description of $R$ is not relevant to what follows.

\begin{proposition}\label{magnetic monopole}
$R$ is symplectomorphic to $\mathcal O(-2)\times \R^2$ (where $\mathcal O(-2) \to \C\P^1$ has symplectic form given by adding the pull-backs of the standard forms via the projection to $\C\P^1$ and the resolution $\mathcal O(-2) \to \C^2/\Z_2$).
\end{proposition}

\begin{proof}[Sketch proof]
As $R$  is a coadjoint orbit of $\SO(4,1)$, the action of $\SO(4,1)$ is Hamiltonian. It follows that $R$ admits a free Hamiltonian $\R^3$-action given as follows: fix a point $p$ at infinity on $H^4$ and consider the subgroup $\R^3 \subset \SO(4,1)$ which acts by linear translations on the horospheres centred at $p$. The action on $R$ admits a section given by considering the restriction of $R \to H^4$ to a geodesic through $p$; this shows that the orbit space is $S^2\times \R$. If the section were Lagrangian, we would be dealing with the cotangent bundle of $S^2 \times \R$, but the $S^2$-factor has symplectic area 1 and so $R$ is symplectomorphic to $\mathcal O(-2) \times \R^2$.
\end{proof} 

\begin{corollary}
$R$ contains no Lagrangian 3-spheres.
\end{corollary}
\begin{proof}
Since $\mathcal O(-2)$ is convex at infinity, this follows from recent work of Welschinger \cite{welschinger} (see, for example, Corollary 4.13).
\end{proof}

\subsubsection{An $\SO(4,1)$-invariant complex volume form}

We now give the proof of Proposition \ref{action on R}, which says that the small resolution admits an $\SO(4,1)$-invariant compatible almost complex structure and invariant complex volume form.

\begin{proof}[Proof of Proposition \ref{action on R}] 
We use the coadjoint orbit description. Lemma \ref{U(2) decomposition} says that at each point $z$ there is a $\U(2)$-equivariant isomorphism of the tangent space $T_z\cong \Lambda^2(\C^2)^* \oplus \C^2$. Accordingly, there is a natural $\U(2)$-invariant almost complex structure on $T_z$ which is $\omega$-compatible and, hence, an $\SO(4,1)$-invariant compatible almost complex structure $J$ on $\mathcal O(\xi)$. 

For the complex volume form, note that $\U(2)$ acts trivially on $\Lambda^3 T^*_z \cong \Lambda^2(\C^2) \otimes \Lambda^2(\C^2)^*$ hence any non-zero element of $\Lambda^3 T_z^*$ can be extended in a unique way to an $\SO(4,1)$-invariant complex volume form.
\end{proof} 

We remark in passing that Lemma \ref{U(2) decomposition} and the almost complex structure used here also have a straightforward interpretation in the twistorial picture. Recall that $R \to H^4$ is a $\C\P^1$-bundle and that $\omega$ is non-degenerate on the fibres. Hence $TR = V \oplus H$ where $V$ is the vertical tangent bundle and $H$ its complement with respect to $\omega$. This corresponds to the splitting in Lemma \ref{U(2) decomposition}. The pseudo-Kähler metric is negative definite on $V$ and positive definite on $H$. Let
$$
J = -J_\mathrm{int}|_V + J_\mathrm{int}|_H
$$
where $J_\mathrm{int}$ is the integrable complex structure on $R \subset \C\P^3$. $J$ is an $\SO(4,1)$-invariant almost complex structure and coincides with that coming from the coadjoint orbit picture. It is an instance of the Eells--Salamon almost complex structure on the twistor space of a Riemannian four-manifold \cite{eells-salamon}. 
 
With this hyperbolic description in hand, we can use hyperbolic four-manifolds to produce symplectic quotients of the small resolution: a hyperbolic four-manifold carries a two-sphere bundle---its twistor space---whose total space is symplectic with trivial canonical bundle. Of course, any compact example will have infinite fundamental group. We will produce a simply connected example in \S\ref{symplectic example} by considering a certain hyperbolic orbifold.

\section{Complex examples}
\label{complex examples}

Compact complex manifolds will be built starting from hyperbolic orbifold metrics on $S^3$ with cone angle $2\pi/m$ along a knot $K \subset S^3$. When such a metric exists, $K$ is said to be $2\pi/m$-hyperbolic. Such knots are well-known to be plentiful. As is explained in \cite{boileau-porti}, it is a consequence of Thurston's orbifold Dehn surgery theorem that when $K$ is a hyperbolic knot---i.e., when $S^3\setminus K$ admits a complete finite-volume hyperbolic metric---$K$ is also $2\pi/m$-hyperbolic for all $m\geq3$ with one sole exception, namely when $K$ is the figure eight knot and $m=3$. There are also infinitely many $\pi$-hyperbolic knots. (We are grateful to Michel Boileau for advice on this matter.)

Let $H^3/\Gamma$ be a hyperbolic orbifold metric on $S^3$ with cone angle $2\pi/m$ along a knot $K$ where $\Gamma\subset \PSL(2,\C)$ is the orbifold fundamental group. It is standard that there a is smooth hyperbolic manifold $M$ which is an $m$-fold cyclic cover $M\to S^3$ branched along the knot $K$. (By Mayer--Vietoris, $H_1(S^3\setminus K) = \Z$; hence there is a homomorphism $\pi_1(S^3\setminus K) \to \Z$ which in turn induces a homomorphism $\psi \colon \Gamma \to \Z_m$; the kernel $\Gamma'$ of $\psi$ has no fixed points on $H^3$ and is the fundamental group of $M$.)

Since $M$ is a hyperbolic manifold, its spin-bundles are complex manifolds with trivial canonical bundle (as is described in \S\ref{S and H^3}). When the spin-structure $P \to M$ is well-chosen, the generator of the $\Z_m$-action lifts to $P$ where it generates a $\Z_{2m}$-action. To produce a simply-connected quotient of $P$,  we then ``twist'' this action around the knot; this induces fixed points in the fibres over $K$ and results in a singular quotient for which all the fibres are simply-connected, giving a simply-connected total space. Finally, we obtain a smooth manifold by taking a crepant resolution.

\subsection{The model singularity and resolution}\label{model singularity}

%\subsubsection{Twisting to produce the model singularity}

We begin by considering the model situation. Take a geodesic $\gamma$ in $H^3$ and let $\Z_m$ act on $H^3$ by fixing $\gamma$ pointwise and rotating perpendicular to $\gamma$ by $2\pi/m$. Let $\varphi =e^{\pi i/m}$. In appropriate coordinates, $\Z_m \subset \PSL(2,\C)$ is generated by the class $[U]$ of the diagonal matrix $U$ with entries $\varphi, \varphi^{-1}$. 

$\Z_m\in \PSL(2,\C)$ is covered by the copy of $\Z_{2m} \subset \SL(2,\C)$ generated by $U$. The action of $\Z_m$ on $H^3$ is covered by the action of $\Z_{2m}$ by right multiplication on $\SL(2,\C)$. However, this action on $\SL(2,\C)$ is free and so the quotient has non-trivial fundamental group. To produce a singular but simply-connected quotient, we ``twist'' to introduce fixed points. This additional twist is given by simultaneously multiplying on the \emph{left}.

Consider the action of $\Z_m$ on $\SL(2,\C)$ generated by conjugation $A \mapsto U^{-1}AU$. Explicitly, in coordinates, the generator~is
\begin{equation}\label{Z_m action}
\left(
\begin{array}{cc}
x & y\\
z & w
\end{array}
\right)
\mapsto
\left(
\begin{array}{cc}
x &\varphi^{-2} y \\
\varphi^2 z & w
\end{array}
\right).
\end{equation}
The points of $\SL(2,\C)$ fixed by $\Z_m$ are the $\C^*$ subgroup of diagonal matrices. On a tangent plane normal to the fixed points, the $\Z_m$-action is that of the $A_m$-singularity (i.e., $\C^2/\Z_m$ with action generated by $U^2 \in \SU(2)$).  Left multiplication on $\SL(2,\C)$ by the $\C^*$ subgroup of diagonal matrices is free and commutes with the action of $\Z_m$; hence it descends to a free $\C^*$-action on the quotient. The $\C^*$-action is transitive on the fixed locus and so gives an identification of the a neighbourhood of the orbifold points with the product of $\C^* \times V$, where $V$ is a neighbourhood of the singular point in the $A_m$-singularity. 

The $\Z_m$-action preserves the fibres of $\SL(2,\C) \to H^3$ and covers the $\Z_m$ action on $H^3$---the fibres are right cosets of $\SU(2)$ and $U\in \SU(2)$---hence there is a projection $\SL(2,\C)/\Z_m\to H^3/\Z_m$. The $\Z_m$-fixed points form a circle bundle over the geodesic $\gamma$. We have implicitly oriented $\gamma$ in our choice of coordinates (\ref{Z_m action}); the image of the fixed points in the frame bundle $\PSL(2,\C)$ are those frames whose first vector is positively tangent to $\gamma$.

Since the $A_m$-singularity admits a crepant resolution, it follows that $\SL(2,\C)/\Z_m$ does too. The exceptional divisor of the resolution $W\to \SL(2,\C)/\Z_m$ maps to the $\C^*$ of singular points with fibre a chain of $m-1$ rational curves over each point, as in the standard crepant resolution of the $A_m$ singularity. 

%\subsubsection{Singularities along closed geodesics}
%\label{model singularity closed loop}

In our model example, the $\Z_m$ fixed locus in $H^3$ is a geodesic $\gamma$. In our compact examples, the branch locus will be a closed geodesic loop. The geodesic $\gamma$ is ``closed up'' by the action of $\Z$ on $\SL(2,\C)$ generated by right-multiplication by the diagonal matrix with entries $a, a^{-1}$, where $a\in \C$ has $|a| > 1$. This action is free and commutes with the $\Z_m$ action described above, hence it induces an action of $\Z$ on $\SL(2,\C)/\Z_m$ and on the crepant resolution $W$. The quotient $W/\Z$ is a crepant resolution of $\SL(2,\C)/(\Z_m\oplus \Z)$, a complex orbifold whose singular locus is an elliptic curve. This is the singularity which will actually appear in our compact examples.

Finally note that the free $\C^*$-action on $\SL(2,\C)/\Z_m$ described above commutes with the action of $\Z$ and so descends to $\SL(2,\C)/(\Z_m\oplus \Z)$. It induces a free $\C^*$-action on the resolution, a feature which will be shared by our compact examples.

\subsection{Compact examples}

%\subsubsection{Complex orbifolds from hyperbolic orbifolds}

We now consider a hyperbolic orbifold metric on $S^3$ with cone angle $2\pi/m$ along a knot $K$. It is a standard fact that there a is smooth hyperbolic manifold $M$ which is an $m$-fold cyclic cover $M\to S^3$ branched along the knot $K$. $\Z_m$ acts by isometries on $M$, fixing the geodesic branch locus pointwise and rotating the normal bundle by $2\pi/m$. The action of $\Z_m$ on the universal cover $H^3$ of $M$ is precisely the situation considered in \S\ref{model singularity}: $\Z_m \subset \PSL(2,\C)$ is generated (in appropriate coordinates) by the class $[U]$ of a matrix $U$ which is diagonal with entries $\varphi$ and $\varphi^{-1}$, where $\varphi = e^{\pi i/m}$. 

In order to produce a complex orbifold we first try to lift the $\Z_m$-action to a spin bundle of $M$. Let $Q$ denote the frame bundle of $M$ and $P$ a choice of spin bundle. The generator of the $\Z_m$-action on $M$ induces a diffeomorphism $f$ of $Q$ and we aim to choose the spin structure so that $f$ lifts to $P$. For this we state the following standard result:

\begin{lemma}
Let $Y$ be a connected manifold, $f\colon Y \to Y$ a diffeomorphism and $Y_h \to Y$ the double cover corresponding to the element $h \in H^1(Y, \Z_2)$. Then $f$ lifts to $Y_h$ if and only if $f^*h = h$.
\end{lemma}

Double covers of $Q$ are parametrised by $H^1(Q,\Z_2)$ and spin structures correspond to elements which are non-zero on restriction to the fibres of $Q\to M$. So to lift the generator $f$ of the $\Z_m$-action to $P$ we must find a suitable invariant $h \in H^1(Q, \Z_2)$. For certain values of $m$, this can always be done.

\begin{lemma}\label{action lifts}
If $m$ is odd or $m=2^r$ then there is a $\Z_m$-invariant element of $H^1(Q, \Z_2)$ corresponding to a spin structure $P \to Q$.
\end{lemma}
\begin{proof}
If $m$ is odd, take any $h \in H^1(Q,\Z_2)$ whose restriction to a fibre is non-zero and average over the group $\Z_m$ to give an invariant element. Since $m$ is odd, the fibrewise restriction remains non-zero.

When $m=2^r$ we use a lemma from knot theory: the $2^r$-fold cover $M' \to S^3$ branched along a knot has $H^1(M, \Z_2) = 0$ (see, e.g., page 16 of \cite{gordon}). It follows that $M$ has a unique spin structure, $H^1(Q, \Z_2)\cong \Z_2$ and the generator of $H^1(Q,\Z_2)$ is $\Z_m$-invariant. 
\end{proof} 

So, when $m$ is odd or $m=2^r$, there is a choice of spin structure $P\to Q$ for which  generator of the $\Z_m$-action lifts. Upstairs in $P$, the lifted action has order $2m$. This can be seen by considering the $\Z_m$-action on a fibre of $Q$ over a point in $K$: here we have lifted the standard action of rotation by $2\pi/m$ on $\SO(3)$ to its double cover $\SU(2)$; it is straightforward to check the order upstairs is $2m$. The action of $\Z_{2m}$ on the universal cover $\SL(2,\C)$ of $P$ is generated (again, in appropriate coordinates) by right multiplication by the matrix $U$ which is diagonal with entries $\varphi$ and $\varphi^{-1}$, where $\varphi = e^{\pi i/m}$. 

Now we can ``twist'' the action around the knot, to introduce a singularity just as in the model case. We define an action of $\Z_m$ on $P$ by sending the generator to conjugation by $U$. More correctly, this defines an action of $\Z_m$ on the universal cover of $P$; but we have only altered the original $\Z_{2m}$-action by left-multiplication and all deck transformations come from right-multiplication, so the $\Z_m$-action on $\SL(2,\C)$ commutes with the action of $\pi_1(M)$ by deck transformations and hence descends to a $\Z_m$-action on $P$. We consider the quotient $P/\Z_m$. Just as the model singularity projects $\SL(2,\C)/\Z_m \to H^3/\Z_m$ there is a projection $P/\Z_m \to S^3$. Away from the knot $K$ this is a locally trivial $S^3$-bundle, over a small neighbourhood of $K$ we have precisely the model singularity considered above in \S\ref{model singularity}. It follows that we can glue in a neighbourhood of the exceptional divisor from the model crepant resolution to obtain a crepant resolution $X\to P/\Z_m$, giving a complex manifold with trivial canonical bundle.

\begin{lemma}\label{C*-action}
The resolution $X$ admits a fixed-point free $\C^*$-action.
\end{lemma}

\begin{proof}
The $\Z_m$-action on $P$ is given (on the universal cover $\SL(2,\C)$) by conjugation by a diagonal matrix with determinant 1. Hence, on the universal cover, it commutes with left-multiplication by the whole $\C^*$ of such diagonal matrices. Since the remaining deck transformations are all given by right-multiplication, they also commute with left-multiplication by $\C^*$, hence this $\C^*$-action descends to $P/\Z_m$ and then lifts to the resolution $X$. The action is free on $X$ since  the same is true of the $\C^*$-action on $\SL(2,\C)$. 
\end{proof}

\subsection{Simple connectivity and non-Kählerity}

To show that our examples are simply connected, we begin with the following standard lemma. 

\begin{lemma}\label{fundamental group lemma}
Let $X$ and $Y$ be two finite dimensional CW complexes and let $f \colon X \to Y$ be a surjective map with connected fibres. Suppose that $Y$ has an open cover by sets $U_i$ such that for any $y \in U_i$ the inclusion homomorphism $\pi_1(f^{-1}(y)) \to \pi_1(f^{-1}(U_i))$ is an isomorphism. Then the following sequence is right-exact:
$$
\pi_1(f^{-1}(y)) \to \pi_1(X) \to \pi_1(Y)\to 0.
$$
\end{lemma}

%\begin{proof}
%Surjectivity of $\pi_1(X)\to \pi_1(Y)$ follows from the fact that $f$ is surjective with connected fibres. We must also prove that if $\gamma$ is a path based at $x\in f^{-1}(y)$ with $f(\gamma)$ contractible, then $\gamma$ is homotopic in $X$ to a path lying entirely in the fibre $f^{-1}(y)$. 

%Since $f(\gamma)$ is contractible, there is a map from the square $[0,1]\times[0,1]$ into $X$ with boundary $f(\gamma)$. Consider a fine grid on the square, of size $1/N$ say, such that each $1/N \times 1/N$ square is contained in one of the $U_i$. Whilst the square needn't lift, the grid does, using connectivity of the fibres, and the lift can be chosen with boundary $\gamma$. Using the lifted grid we can write $\gamma$ as the product of $N^2$ paths, each of which follows some path on the lifted grid, goes round a small square and then comes back along the same path. Each small square lies in $f^{1-}(U_i)$ for some $U_i$ hence each of these $N^2$ paths is homotopic to a path contained entirely in the fibre $f^{-1}(y)$.
%\end{proof}

\begin{lemma}\label{simply-connected}
Let $X$ be the complex threefold associated to a $2\pi/m$-hyperbolic knot, as described above. Then $X$ is simply-connected.
\end{lemma}

\begin{proof}
First, we consider the projection $f \colon P/\Z_m \to S^3$. Away from the knot $K$ this is a locally trivial $S^3$-fibration. The fibre of $f$ over a point in $K$ is the quotient $\SU(2)/\Z_m$ via the action (\ref{Z_m action}) which is easily seen to be homeomorphic to $S^3$. Hence, to apply Lemma \ref{fundamental group lemma} to $f$, we just need to check that each point of the knot $K$ is contained in an open set $U \subset S^3$ such that $\pi_1(f^{-1}(U))=0$. It suffices to do this in the local model of $\Z_m$ acting on $H^3$ fixing a geodesic (as in \S\ref{model singularity}). Pick a small geodesic ball $B$ centred at a point $p$ on the fixed geodesic. Parallel transport in the radial directions of $B$ gives a retraction of the frame bundle over $B$ to the fibre over $p$ and, hence, a retraction of the portion $S|_B$ of $\SL(2,\C)$ lying over $B$ to the copy of $\SU(2)$ over $p$. This retraction is $\Z_m$-equivariant, hence the portion $(S|_B)/\Z_m$ of $\SL(2,\C)/\Z_m$ lying over $B/\Z_m$ retracts onto $\SU(2)/\Z_m \cong S^3$. In particular, it is simply connected as required.

It follows that $\pi_1(P/\Z_m) = 0$. To deduce that the resolution $X$ is also simply connected we apply Lemma \ref{fundamental group lemma} again, this time to the map $r \colon X \to P/\Z_m$. Away from the singular locus, $r$ is one-to-one; meanwhile each point on the singular locus has preimage a chain of $m-1$ copies of $S^2$. So to apply the lemma we must show that each point of the singular locus has a neighbourhood $U$ for which $\pi_1(r^{-1}(U))=0$. This follows from the fact that, near any point in the singular locus of $P/\Z_m$, the singularity looks locally like the product $D \times (\C^2/\Z_m)$ of a disc with the $A_m$-singularity and the resolution looks locally like the product $D \times A_m$ of a disc with the $A_m$-resolution. 
\end{proof}

\begin{lemma}
There is no compatible Kähler structure on the complex threefold associated to a $2\pi/m$-hyperbolic knot.
\end{lemma}

\begin{proof}
The fact that our examples are simply-connected (Lemma \ref{simply-connected}) and admit a free $\C^*$-action (Lemma \ref{C*-action}) implies that they have no compatible Kähler metric. For if they admitted a compatible Kähler metric, by averaging it would be possible to find an $S^1\subset \C^*$ invariant Kähler form. Now $b_1=0$ implies that the symplectic $S^1$-action would, in fact, be Hamiltonian and hence have fixed points. 
\end{proof}

\subsection{Diffeomorphism type for $\pi$-hyperbolic knots}
\label{diffeo for pi-hyperbolic knots}

We now turn to the question of the diffeomorphism type of our examples in the topologically most simple case, that of a $\pi$-hyperbolic knot. We proceed via Wall's classification theorem \cite{wall}. Wall's result states that oriented, smooth, simply-connected, spin, 6-manifolds with torsion free cohomology are determined up to oriented-diffeo\-morphism~by:
\begin{itemize}
\item 
the integer $b_3$; 
\item
the symmetric trilinear map $H^2 \times H^2 \times H^2 \to \Z$ given by cup-product
\item 
the homomorphism $H^2 \to \Z$ given by cup-product with the first Pontrjagin class.
\end{itemize}

The goal of this section is to prove

\begin{theorem}\label{topology of resolution}
Given a $\pi$-hyperbolic knot, the resulting complex threefold constructed above is diffeomorphic to $2(S^3 \times S^3) \#(S^2 \times S^4)$.
\end{theorem}

To compute the cohomology of the complex threefold, we begin with the topology of the orbifold $P/\Z_2$. In fact, the results here hold for $P/\Z_m$ with any choice of $m$.

Let $K \subset S^3$ be a $2\pi/m$-hyperbolic knot (with $m$ odd or $m=2^r$) and let $U$ be a small tubular neighbourhood of $K$. We write the boundary of $U$ as $S^1_1 \times S^1_2$ where $S^1_1$ is a meridian circle, which is contractible in $U$, and $S^1_2$ is a longitudinal circle, which is contractible in $S^3\setminus K$. Let $f\colon P/\Z_m \to S^3$ denote the projection from the complex orbifold to the hyperbolic orbifold. Write $X_1 = f^{-1}(S^3\setminus K)$ and $X_2 = f^{-1}(U)$.

In what follows, (co)homology groups are taken with coefficients in  $\Z$ unless explicitly stated. We begin with a couple of standard topological lemmas.

\begin{lemma}
Given any knot $K \subset S^3$, the knot complement has $H_1(S^3\setminus K) \cong \Z$, generated by the class of $S^1_1$, whilst $H_2(S^3\setminus K) \cong 0 \cong H_3(S^3 \setminus K)$.
\end{lemma}

\begin{lemma}
Every $\SO(4)$-bundle over a 3-manifold with $H_2 \cong 0 \cong H_3$ is trivial.
\end{lemma}

We apply these results to show that for $j=1,2$, the subset $f^{-1}(S^1_j)$ carries all the homology of $X_j$.

\begin{lemma}\label{product}
$f \colon X_1 \to S^3\setminus K$ is a trivial $S^3$-fibration. The inclusion $f^{-1}(S^1_1) \to X_1$ induces an isomorphism on homology.
\end{lemma}

\begin{proof}
$\SL(2,\C) \to H^3$ is an $\SU(2)$-bundle so, in particular is an $S^3$-bundle with structure group $\SO(4)$. Moreover, this $\SO(4)$ structure is preserved by the image of $\SU(2) \times \SL(2,\C)$ in $\SO(4,\C)$.  Accordingly, away from $K$, the map $f \colon X_1 \to S^3\setminus K$ has structure group $\SO(4)$. The result now follows from the previous two lemmas.
\end{proof}

\begin{lemma}\label{X2 retracts}
$X_2$ retracts to $f^{-1}(K)$. The inclusion $f^{-1}(S^1_2) \to X_2$ induces an isomorphism on homology.
\end{lemma}

\begin{proof}
To prove that $X_2$ retracts as claimed it suffices to consider the model of \S\ref{model singularity}, $\SL(2,\C)/(\Z\oplus \Z_m)$. The orthogonal retraction of $H^3$ onto the $(\Z\oplus \Z_m)$-fixed geodesic induces a retraction of a small tubular neighbourhood of the geodesic loop in the quotient $H^3/(\Z\oplus \Z_m)$. This lifts to give the claimed retraction of the pre-image of the tubular neighbourhood in $\SL(2,\C)/(\Z\oplus \Z_m)$. It follows from this that the embedding $f^{-1}(S^1_2) \to X_2$  induces an isomorphism on homology. 
\end{proof}

Now we are in a position to compute the homology of the complex orbifold.

\begin{proposition}\label{integral homology of orbifold}
$P/\Z_m$ has the integral homology of $S^3 \times S^3$.
\end{proposition}

\begin{proof}
We compute the cohomology of $P/\Z_m$ by applying the Mayer--Vietoris sequence to the pair $(X_1, X_2)$. For this we need the maps $(i_1)_*$ and $(i_2)_*$ induced on homology by the inclusions $i_1\colon X_1\cap X_2 \to X_1$ and $i_2 \colon X_1 \cap X_2\to X_2$. It follows from Lemma \ref{product} that $X_1 \cap X_2$ is the product $f^{-1}(U\setminus K) \times S^3$. In particular, it retracts to $f^{-1}(S^1_1 \times S^1_2) = S^1_1 \times S^1_2 \times S^3$. 

Let $i'_1$ and $i'_2$ denote the compositions of the retraction of $X_1 \cap X_2$ to $S^1_1 \times S^1_2 \times S^3$ with a consecutive projection to $S^1_1 \times S^3$ and $S^1_2 \times S^3$ respectively. The preceding Lemmas \ref{product} and \ref{X2 retracts} show that the images of the maps $(i_1)_*$ and $(i_2)_*$ coincide with the images of the maps $(i'_1)_*$ and $(i'_2)_*$ respectively. It follows from this that the Mayer--Vietoris sequence of the pair $(X_1,X_2)$ can be identified with the Mayer--Vietoris sequence of the pair $((S^3 \setminus K) \times S^3, U \times S^3)$. Since this second sequence calculates the homology of $S^3\times S^3$, the proposition is proved. 
\end{proof}

The singular locus of $P/\Z_m$ is an elliptic curve $C$. The next step is to compute the homology of the pair $(P/\Z_m, C)$.

\begin{lemma}\label{relative homology}
The non-zero relative homology groups $H_j(P/\Z_m, C)$ are 
$$
H_2(P/\Z_m, C) \cong\Z^2,
\quad
H_3(P/\Z_m, C) \cong \Z^3,
\quad
H_6(P/\Z_m, C) \cong \Z.
$$
\end{lemma}

\begin{proof}
This follows immediately from the exact sequence of the pair $(P/\Z_m, C)$ along with the fact that $P/\Z_m$ has the integral homology of $S^3 \times S^3$.
\end{proof}

To convert this into information about the homology of the resolution $X$ we restrict to the case when $K$ is $\pi$-hyperbolic. 
 
\begin{proposition}\label{integral homology of resolution}
Let $X$ be the complex threefold constructed from a $\pi$-hyperbolic knot, as described above. Then $X$ has the integral homology of $2(S^3 \times S^3)\#(S^2 \times S^4)$.
\end{proposition}

\begin{proof}
Let $E \subset X$ denote the exceptional divisor of the resolution $X \to P/\Z_2$. The singularity in $P/\Z_2$ is locally the product of the elliptic curve $C$ with the $A_2$-singularity in $\C^2$. So $E\cong \C\P^1 \times C$, with normal bundle $\mathcal O(-2)$ pulled back from $\C\P^1$
 
Since $X/E$ (where $E$ is crushed to a point) is homeomorphic to $(P/\Z_2)/C$, it follows that the relative groups $H_j(X,E) \cong H_j(P/\Z_2, C)$ are given by Lemma~\ref{relative homology}. We now consider the long exact sequence of the pair $(Z,E)$. Since $H_5(E) = 0=H_5(Z,E)$, we have that $H_5(Z) = 0$. Meanwhile, since $H_5(Z,E) = 0=H_4(Z,E)$ we have that $H_4(Z) \cong H_4(E) \cong \Z$. 

From here we can compute all the Betti numbers. Indeed, by Poincaré duality, $b_2(Z) = b_4(Z) = 1$; now considering the alternating sum of the ranks of the groups in the part of the sequence from $H_4(Z,E) = 0$ to $H_1(Z) = 0$ gives $b_3 =4$. So the rational homology of $Z$ coincides with that of the connected sum. It remains to show that $H_2(Z)$ and $H_3(Z)$ are torsion-free. 

Since $H_4(Z,E) = 0$, the long exact sequence gives
$$
0
\to 
H_3(E) 
\to 
H_3(Z) 
\stackrel{j}{\to} 
H_3(Z,E) \to \cdots
$$
As $H_3(Z,E) \cong \Z^3$ is torsion-free, all torsion in $H_3(Z)$ must be contained in $\ker j$. However, $\ker j \cong H_3(E) \cong \Z^2$ is torsion-free, hence so is $H_3(Z)$.

The argument for $H_2(Z)$ is more involved. First, note that since $H_1(Z) = 0$ and $H_2(Z,E) \cong \Z^2 \cong H_1(E)$, the map $H_2(Z,E) \to H_1(E)$ in the sequence of the pair $(Z,E)$ is an isomorphism. Now the preceding part of the sequence gives that $H_2(E) \to H_2(Z)$ is surjective. $E \cong C \times \C\P^1$, so $H_2(E) \cong \Z^2$ is generated by the class of a $\C\P^1$-fibre and the class of a $C$-fibre. Since $H_2(Z)$ has rank 1, to show that $H_2(Z) \cong \Z$ it suffices to show that the $C$-fibre in $E$ is null-homologous in $Z$. 

For this it suffices to consider the model case of the resolution $X'\to \SL(2,\C)/(\Z_2 \oplus \Z)$ as in \S\ref{model singularity}. This is because if there is a three-chain bounding $C$ in $X'$, then the retraction in $H^3$ onto the fixed geodesic pushes this 3-chain into an arbitrarily small neighbourhood of the exceptional divisor; hence there is such a 3-chain in the resolution~$X$. 

$\SL(2,\C)/(\Z_2\oplus\Z)$ is the quotient of $\SL(2,\C)/\Z$ by an involution with fixed locus an elliptic curve $C$. Away from $C$, there is a two-to-one map $\SL(2,\C)/\Z \dashrightarrow X'$. This extends to the blow-up $\hat X$ of $\SL(2,\C)/\Z$ along $C$ to give a ramified double cover $\hat X \to X'$. (This is analogous to the fact that, for bundles over $\C\P^1$, squaring $\mathcal O(-1) \to \mathcal O(-2)$ gives a ramified double cover of the $A_2$-resolution.) The ramification locus in $\hat X$ is the exceptional divisor $\hat E$ of the blow up $\hat X \to \SL(2,\C)/\Z$. $\hat E$  is identified with branch locus in $X'$, which is the exceptional divisor $E$ of the resolution $X' \to \SL(2,\C)/(\Z_2 \oplus \Z)$. So $\hat E \cong \C\P^1 \times C$ and to show that a $C$-fibre of $E$ is null-homologous in $X'$ it suffices to show that a $C$-fibre of $\hat E$ is null-homologous in $\hat X$.

To prove this statement, first note that $\SL(2,\C)/\Z$ is homeomorphic to $S^3 \times S^1 \times \R^2$. The elliptic curve $C$ corresponds to taking the product of a Hopf circle $S^1 \subset S^3$  with $S^1\times \{\mathrm{pt}\}$. Choosing a different Hopf circle gives another copy $C'$ of $C$ which lies entirely inside $(\SL(2,\C)/\Z)\setminus C$. On the one hand, $C'$ is null-homologous in $(\SL(2,\C)/\Z)\setminus C$ (simply move it in the $\R^2$-direction so that it lies in a complete copy of $S^3\times S^1$). On the other hand, when thought of as a 2-cycle in the blow-up $\hat X$, $C'$ is homologous to a $C$-fibre of $\hat E$. Hence the $C$-fibre of $\hat E$ is zero in homology and the proposition is proved.
\end{proof}

\begin{lemma}
Let $X$ be the complex threefold constructed from a $\pi$-hyperbolic knot, as described above. Then all the Chern classes of $X$ vanish.
\end{lemma}
\begin{proof}
We already know that $c_1=0$ and $c_3= 0$ (as it is the Euler characteristic). Since $H_4 (X)$ is generated by the exceptional divisor $E$, to prove $c_2=0$ it suffices to show that $\langle c_2, E\rangle = 0$. Now, the normal bundle to $E \cong \C\P^1 \times C$ is $\mathcal O(-2)$ pulled-back from $\C\P^1$. Combining this with the fact that $c_1(X) = 0$ gives that over $E$ we have a topological isomorphism $TX|_E \cong \mathcal O(-2)\oplus \mathcal O(2)\oplus \mathcal O$. Since all of these bundles are pulled back to $E$ from $\C\P^1$ we see that $\langle c_2, E\rangle = 0$, as claimed.
\end{proof}   

We are now ready to prove that $X$ is diffeomorphic to $2(S^3\times S^3)\#(S^2 \times S^4)$. 

\begin{proof}[Proof of Theorem \ref{topology of resolution}]
In order to apply Wall's Theorem we must first check that $X$ is spin, which follows from the fact that it is complex with trivial canonical bundle.   

Next, we must show the cup-product is trivial on $H^2$. In the proof of Proposition \ref{integral homology of resolution}, we saw that $H^2(X)$ is generated by the Poincaré dual $e$ to the exceptional divisor $E$. Since $H_4(X)$ is generated by $E$, it suffices to check that $\langle e^2, E \rangle =0$. The normal bundle of $E \cong \C\P^1 \times C$ is pulled back from $\C\P^1$. Hence, when restricted to $E$, $e$ is also pulled back from $\C\P^1$ and so squares to zero. 

Finally, we must show that $p_1(X) = 0$, but this follows from the formula $p_1 = c_1^2 - 2c_2$ combined with $c_1=0=c_2$.
\end{proof}

\subsection{Recovering the orbifold from the threefold}

In this section we prove that the hyperbolic orbifold can be recovered from the complex threefold, so distinct orbifolds lead to distinct threefolds. 

\begin{theorem}
Let $X$ and $X'$ be the complex threefolds constructed from hyperbolic orbifolds $N$ and $N'$ as above. If $X$ and $X'$ are biholomorphic then $N$ and $N'$ are isometric.
\end{theorem}

\begin{proof}
Suppose that $X$ is constructed from an orbifold metric on $S^3$ with cone angle $2\pi/m$, whilst $X'$ involves the cone angle $2\pi/m'$. Away from the elliptic curve $C \subset P/\Z_m$ of singular points, $r \colon X \to P/\Z_m$ is an isomorphism. The exceptional locus $E = r^{-1}(C)$ is biholomorphic to the product of $C$ with a chain $\Theta_m$ of $m-1$ $\C\P^1$s as in the $A_m$-resolution. 

First, we claim that any analytic surface in $X$ is contained in $E$. Any surface not contained in $E$ would project to a surface in $P/\Z_m$; its preimage would then be an analytic surface in $P = \SL(2,\C)$. However, it is known that the quotient of $\SL(2,\C)$ by a cocompact subgroup never contains an analytic surface (Theorem 9.2 of \cite{huckleberry-winkelmann}). It follows that $E$ is the union of all surfaces in $X$. Since $X$ and $X'$ are biholomorphic, the union of all surfaces in $X'$ is also biholomorphic to $C \times \Theta_m$. It follows that the exceptional loci $E',  E$ are biholomorphic and so $m = m'$.

Next, note that  a neighbourhood of $E$ is \emph{canonically} biholomorphic to the product of $C$ with a neighbourhood of $\Theta_m$ in the $A_m$-resolution. The canonical biholomorphism is provided by the $\C^*$-action on $P/\Z_m$ which acts by translations on $C$ and which lifts to $X$. It follows that there is a canonical choice of contraction of the exceptional loci $E$ and $E'$ which produces $P/\Z_m$ and $P'/\Z_m$. Since $X$ and $X'$ are biholomorphic, $P/\Z_m \cong P'/\Z_m$. By construction, the orbifold universal cover of $P/\Z_m$ is $\SL(2,\C)$ and that of $N$ is $H^3$ and both are obtained as quotients by the action of the same group. Hence the orbifold fundamental groups of $P/\Z_m$ and $N$ coincide.  So $P /\Z_m \cong P'/\Z_m$ induces an isomorphism between the orbifold fundamental groups of $N$ and $N'$ and hence, by Mostow rigidity, an isometry between $N$ and $N'$.\end{proof}

\section{A symplectic example}
\label{symplectic example}

In this section we will explain how a similar approach---passing from a hyperbolic orbifold to a symplectic manifold via a crepant resolution---can be used to build a simply-connected symplectic manifold with $c_1=0$ which does not admit a compatible Kähler structure. 

We content ourselves here with a single example for which it is not hard to find a crepant resolution because of the simple nature of the singularities. In \S\ref{concluding remarks} we describe an infinite sequence of hyperbolic 4-orbifolds with slightly more complicated singularities to which this procedure could be applied. 

\subsection{The Davis manifold}

Our construction is based on a beautiful hyperbolic 4-manifold called the Davis manifold $M$ (see \cite{davis} as well as the description in \cite{ratcliffe-tschantz} which also computes the homology groups of $M$). The key fact for us is that $M$ admits an isometric involution which kills the fundamental group. This is analogous to the role played in the Kummer construction by the involution $z \mapsto -z$ on an abelian surface.

$M$ is built using a regular polytope called the 120-cell (or hecatonicosachoron). The 120-cell is a four-dimensional regular solid with 120 three-dimensional faces, the ``cells'', each of which is a solid dodecahedron. Each edge is shared by 3 dodecahedra and each vertex by 4 dodecahedra. In total, the 120-cell has 600 vertices, 1200 edges and 720 pentagonal faces. Take a hyperbolic copy $P\subset H^4$ of the 120-cell in which the dihedral angles are $2\pi/5$. For each pair of opposite dodecahedral faces of $P$ there is a unique hyperbolic reflection which identifies them. Gluing opposite faces via these reflections gives the hyperbolic four-manifold~$M$. 

The central involution of $H^4$ which fixes the centre of $P$ preserves both $P$ and the identifications of opposite faces, hence it gives an isometric involution $\sigma$ of $M$. Our symplectic construction will begin with the resulting orbifold $M/\sigma$, which we call the \emph{Davis orbifold}. 

To analyse the fixed points of $\sigma$ it is helpful to use the so-called ``inside-out'' isometry of $M$ (defined in \cite{ratcliffe-tschantz}). To describe this, note that $P$ can be divided up into 14400 hyperbolic Coxeter simplexes. The vertices of a simplex are given by taking first the centre of $P$, then the centre of one of its 120 3-faces $F$, then the centre of one of the 12 2-faces $f$ of $F$, then the centre of one of the 5 edges $e$ of $f$ and, finally, one of the two vertices of $e$. Denote by $v_1 , v_2 , v_3 , v_4 , v_5$ one such choice. The corresponding simplex has a isometry that exchanges $v_1$ (the centre of $P$) with $v_5$ (a vertex of $e$), $v_2$ (the centre of $F$) with $v_4$ (the centre of an edge of $F'$) and fixes $v_3$ (the centre of $f$). This isometry of the simplex extends to define the inside-out isometry of $M$, which commutes with~$\sigma$. 

\begin{lemma}
The fixed set of $\sigma$ consists of $122$ points. The quotient $M/\sigma$ is simply connected as a topological space.
\end{lemma}

\begin{proof}
In the interior of the 120-cell there is only one fixed point, the centre, all other fixed points of $\sigma$ lie on the image in $M$ of the boundary of $P$. Let $F$ denote the image in $M$ of a three-dimensional face of $P$; $\sigma$ preserves $F$ and induces on it the symmetry of the dodecahedron given by inversion $x \mapsto -x$ with respect to its centre. So, once again, in the interior of $F$ there is only one fixed point, its centre. Considering all opposite pairs of three-dimensional faces of $M$ this gives 60 more fixed points of $\sigma$. All remaining fixed points are contained in the image in $M$ of the union of the 2-faces of $P$. 

The symmetry $\sigma$ takes 2-faces to 2-faces. We claim next that $\sigma$ does not fix an interior point of any pentagonal 2-face. Assume for a moment that it does fix such a point. Then it would give an involution of the pentagon which would hence fix a vertex and so also the line joining the vertex to the centre of the polygon. The Davis manifold has two distinguished points, the centre and the image of all the vertices of the 120-cell. The assumption that $\sigma$ fixes an interior point of a pentagonal 2-face gives a $\sigma$-fixed tangent direction at the vertex point in $M$. However, the inside-out isometry exchanges the centre and vertex of $M$. Since $\sigma$ acts as $x \mapsto -x$ at the centre it does so also at the vertex and hence acts freely on the tangent space there. It follows that $\sigma$ does not fix an interior point of any 2-face.

The remaining fixed points are contained in the image in $M$ of the union of the edges of $P$. Under the inside-out involution of $M$, the middles of all edges are exchanged with centres of all 3-faces  whilst the centre of $P$ is exchanged with the image in $M$ of the vertices of $P$. Since the inside-out isometry commutes with $\sigma$, this give an additional 61 fixed points of $\sigma$ making 122 in total.

We now turn to the (topological, not orbifold) fundamental group $\pi_1(M/\sigma)$. The map $\pi_1(M) \to \pi_1(M/\sigma)$ is surjective so we need to show its image is trivial. Consider the 60 closed geodesics $\gamma_i$ in $M$ going through the centre of $P$ and joining the centres of opposite faces. The deck transformations corresponding to these geodesics generate the whole of $\pi_1(M)$. Indeed, these deck-transformations take the fundamental domain $P$ to all its 120 neighbours. Now the result follows from the fact that every loop $\sigma(\gamma_i)$ is contractible.
\end{proof}

\subsection{The model singularity and resolution}

Locally, the singularities of $M/\sigma$ are modelled on the quotient of $H^4$ by $x\mapsto -x$. (Here $x$ is the coordinate provided by the Poincaré ball model of $H^4$.) From Proposition \ref{action on R} we know that the corresponding symplectomorphism of $R$ is given by $z\mapsto -z$ in the vector-bundle fibres of $\mathcal O(-1) \oplus \mathcal O(-1)$. The quotient $\mathcal O(-1)\oplus \mathcal O(-1)/\Z_2$ is a Kähler Calabi--Yau orbifold with singular locus $\C\P^1$ corresponding to the zero section. We next describe a crepant resolution of this singularity.

\begin{lemma}\label{twistor resolution}
There is a crepant resolution 
$$
\mathcal O(-2,-2) \to \mathcal O(-1)\oplus \mathcal O(-1)/\Z_2
$$
where $\mathcal O(-2,-2)\to \C\P^1\times \C\P^1$ is the tensor product of the two line bundles given by pulling back $\mathcal O(-2) \to \C\P^1$ from either factor.
\end{lemma}
\begin{proof}
Blow up the zero section of $\mathcal O(-1)\oplus \mathcal O(-1)$ to obtain the total space of $\mathcal O(-1,-1) \to \C\P^1 \times \C\P^1$. The $\Z_2$-action lifts to this line bundle where it again has fixed locus the zero section and acts by $z \mapsto -z$ in the fibres. For such an involution on \emph{any} line bundle $L$, the square gives a resolution $L^2 \to L/\Z_2$. Hence 
$$
\mathcal O(-2,-2)\to \mathcal O(-1,-1)/\Z_2 \to \mathcal O(-1)\oplus \mathcal O(-1)/\Z_2
$$ 
gives the claimed resolution
\end{proof}

It must be emphasised that this is resolution is \emph{holomorphic}. The total space of $\mathcal O(-2,-2)$ is a Kähler manifold with trivial canonical bundle and, away from the exceptional divisor, the map in Lemma   \ref{twistor resolution} is a biholomorphism when we consider $\mathcal O(-1)\oplus\mathcal O(-1)/\Z_2$ with its \emph{holomorphic} complex structure. However, when constructing symplectic six-manifolds from hyperbolic four-manifolds, the relevant almost complex structure and volume form on $\mathcal O(-1) \oplus \mathcal O(-1)$ (and its $\Z_2$-quotient) are not the holomorphic ones; rather we use the $\SO(4,1)$-invariant almost complex structure and complex volume form from Proposition \ref{action on R}. Lemma \ref{twistor resolution} can only be used to provide crepant resolutions of Calabi--Yau singularities modelled on the \emph{holomorphic} geometry of $\mathcal O(-1)\oplus \mathcal O(-1)/\Z_2$ and \emph{not} the $\SO(4,1)$-invariant almost complex structure and complex volume form of Proposition \ref{action on R}.

So, in order to apply Lemma \ref{twistor resolution} to resolve singularities in hyperbolic twistor spaces, we need to interpolate between the holomorphic structures near the zero section in $\mathcal O(-1)\oplus \mathcal O(-1)$ to the  $\SO(4,1)$-invariant structures outside a small neighbourhood of the zero section. This interpolation is provided by the following result.

\begin{lemma}\label{interpolation}
Let $R_\delta$ denote the part of $R$ lying over a geodesic ball in $H^4$ of radius $\delta$. For any $\delta >0$,  there is an $\SO(4)$-invariant compatible almost complex structure $J$ on $R$ and an $\SO(4)$-invariant nowhere-vanishing section $\Omega$ of the $J$-canonical-bundle such that:
\begin{itemize}
\item Over $R_\delta$, $J$ and $\Omega$ agree with the standard holomorphic structures.
\item Over $R \setminus R_{2\delta}$, $J$ and $\Omega$ agree with the $\SO(4,1)$-invariant structures from Proposition \ref{action on R}
\end{itemize}
\end{lemma}

\begin{proof}
As is standard, an $\SO(4)$-invariant interpolation between the ``inside'' and ``outside'' Hermitian metrics gives the existence of $J$.

To produce $\Omega$ we start with a description of the $\SO(4)$-action away from the zero-section $R_0$. The stabiliser of a point $p \in R\setminus R_0$ is a circle $S^1_p \subset \SO(4)$ and the orbit of $p$ is 5-dimensional (in fact, isomorphic as an $\SO(4)$-space to the unit tangent bundle of $S^3$). The lift of a geodesic ray out of the origin in $H^4$ meets each $\SO(4)$-orbit in a unique point, giving a section for the action. We interpolate between the holomorphic and hyperbolic complex volume forms along the relevant portion of this lifted ray and then use the $\SO(4)$-action to extend the resulting $3$-form to the whole of $R$. In order for this to work it is sufficient that at every point $p\in R\setminus R_0$ the action of $S^1_p$ on the fibre of the $J$-canonical-bundle at $p$ is trivial. But since the weight is integer valued and continuous it is constant on $R\setminus R_0$ so we can compute it for some $p$ outside of  $R_{2r}$ where everything agrees with the hyperbolic picture. Here we already have an $\SO(4)$-invariant (hence $S^1_p$-invariant) complex volume-form so the weight is zero as required.
\end{proof}

\subsection{The twistor space of the Davis orbifold}

With Lemmas \ref{twistor resolution} and \ref{interpolation} in hand, we can now take a crepant resolution of the twistor space of the Davis orbifold $M/\sigma$. Let $Z \to M$ denote the twistor space of the Davis manifold The involution $\sigma$ lifts to an involution of $Z$ which we still denote $\sigma$. $Z/\sigma$ is a symplectic orbifold  with singularities along 122 $\C\P^1$s, each modelled on $\mathcal O(-1)\oplus \mathcal O(-1)/\Z_2$. 

Let $\delta$ be a positive number small enough that the geodesic balls in $M$ of radius $2\delta$ centred on the $\sigma$-fixed points are embedded and disjoint. Then, by Lemma \ref{interpolation}, on $Z$ we can find a new almost complex structure $J$ and complex volume form $\Omega$ such that outside the geodesic $2\delta$-balls they agree with the hyperbolic structures coming from Proposition \ref{action on R}, whilst inside the balls of radius $\delta$ they agree with the holomorphic structures coming from the holomorphic geometry of $\mathcal O(-1)\oplus \mathcal O(-1)$. It follows from the $\SO(4)$-invariance in Lemma \ref{interpolation} that $J$ and $\Omega$ are $\sigma$-invariant. 

In this way the quotient $Z/\sigma$ is a symplectic orbifold with an almost complex structure and complex volume form which are modelled near the singular curves on the holomorphic geometry of $\mathcal O(-1)\oplus \mathcal O( -1)/\Z_2$. It follows from Lemma \ref{twistor resolution} that there is a resolution $\hat Z \to Z/\sigma$ in which the singular curves have been replaced by copies of $\C\P^1\times \C\P^1$ with normal bundle $\mathcal O(-2,-2)$; moreover, $\hat Z$ carries an almost complex structure $\hat J$ and complex volume form $\hat \Omega$ so that $c_1(\hat Z,\hat J)=0$. 

Finally we need to define the symplectic structure on $\hat Z$. Pulling back the symplectic form via $\hat Z \to Z$ gives a symplectic form on the complement of the exceptional divisors. To extend it we use a standard fact about resolutions in Kähler geometry. Given any neighbourhood $U$ of the zero locus in $\mathcal O(-2,-2)$, there is a Kähler metric on $\mathcal O(-2,-2)$ for which the projection to $\mathcal O(-1)\oplus \mathcal O(-1)/\Z_2$ is an isometry on the complement of $U$. (This amounts to the fact that the zero locus has negative normal bundle.) 

So, in the model, the pull-back of the symplectic form extends over the exceptional divisor in a way compatible with holomorphic complex structure. Taking $U$ sufficiently small  and doing this near all 122 exceptional divisors defines a symplectic form $\omega$ on $\hat Z$ which is compatible with $\hat J$. 

\subsection{Simple-connectivity and non-Kählerity}

This section proves that $\hat Z$ is simply connected and admits no compatible integrable complex structure. The second fact will follow from the first and the fact that $b_3(\hat Z)=0$.

\begin{lemma}
$\hat Z$ is simply connected.
\end{lemma}
\begin{proof}
We first apply Lemma \ref{fundamental group lemma} to the map $Z/\sigma \to M/\sigma$. The fibres are $S^2$s and we see that $\pi_1(Z/\sigma) = 1$. Next we apply Lemma \ref{fundamental group lemma} to $\hat Z \to Z/\sigma$. This time the fibres are points or $S^2$s and we deduce that $\pi_1(\hat Z) =1$.
\end{proof}

To prove that $b_3(\hat Z)=0$ we invoke a lemma of McDuff on the cohomology of manifolds obtained by symplectic blow-ups. 

\begin{lemma}[McDuff \cite{mcduff2}]\label{mcduffs lemma}
Let $X$ be a symplectic manifold and $C \subset X$ a smooth symplectic submanifold of codimension $2k$. Let $\tilde X$ denote the blow-up of $X$ along $C$. Then the real cohomology of $\tilde X$ fits into a short exact sequence of graded vector spaces
$$
0 \to H^*(X) \to H^*(\tilde X) \to A^* \to 0
$$
where the first arrow is pull-back via $\tilde X \to X$ and where $A^*$ is free module over $H^*(C)$ with one generator in each dimension $2j$, $1\leq j \leq k-1$.
\end{lemma}

\begin{lemma}
$b_3(\hat Z) = 0$.
\end{lemma}

\begin{proof}
Recall that $Z \to M$ is the twistor space of the Davis manifold. We first blow up the 122 fibres which lie over the fixed points of $\sigma$ to obtain the new manifold $\tilde Z$. It follows from Lemma \ref{mcduffs lemma} that pulling back cohomology via $\tilde Z \to Z$ induces an isomorphism $H^3(\tilde Z) \cong H^3(Z)$.

Next, notice that $\sigma$ lifts to $\tilde Z$ and that $\hat Z = \tilde Z/\sigma$. We now show that $\sigma$ acts as $-1$ on $H^3(\tilde Z)$. To see this, consider the action of $\sigma$ on the Davis manifold $M$. It acts on $H^1(M)$ as $-1$ and hence also as $-1$ on $H^3(M)$. Now $Z \to M$ is a sphere-bundle so, by Leray--Hirsch, $H^*(Z)$ is a free module over $H^*(M)$ with a single generator in degree 2 corresponding to the first Chern class of the vertical tangent bundle. This generator is preserved by $\sigma$, so $\sigma$ acts as $-1$ on $H^3(Z)$ and hence also as $-1$ on $H^3(\tilde Z)$. From this we deduce that $H^3(\hat Z) =0$. For if it contained a non-zero element, the pull-back to $\tilde Z$ would be a $\sigma$-invariant element of $H^3(\tilde Z)$.
\end{proof}

\begin{corollary}
There is no K\"ahler structure on $\hat Z$ with $c_1=0$. In particular, the symplectic structure on $\hat Z$ described above admits no compatible complex structure.
\end{corollary}
\begin{proof}
For a Kähler manifold, the vanishing of $b_1$ implies the Picard torus is trivial. Now $c_1=0$ implies the existence of a holomorphic volume form, hence $b_3 \geq 2$. 
\end{proof}

Note that whilst we have shown that there is no Calabi--Yau Kähler structure on $\hat Z$, it is not clear, to us at least, whether or not $\hat Z$ admits a Kähler structure when one does not place a restriction on $c_1$.

\section{Some symplectic ``Fano'' manifolds}
\label{symplectic fanos}

In this section we explain how hyperbolic and complex-hyperbolic geometry in higher dimensions leads to symplectic manifolds for which the first Chern class is a positive multiple of the symplectic class, non-Kähler analogues of Fano manifolds.  

\subsection{Hyperbolic geometry in even dimensions}

The passage from hyperbolic 4-manifolds to symplectic 6-manifolds can be generalised to every even dimension, with hyperbolic $2n$-manifolds giving symplectic $n(n+1)$-manifolds. This was first explained via twistors, by Reznikov \cite{reznikov} (although Reznikov did not consider the first Chern class of his examples). Here we give an alternative description in terms of coadjoint orbits, as in \S\ref{coadjoint description}.

The Lie algebra $\so(2n,1)$ consists of $(2n+1)\times (2n+1)$ matrices of the form
\begin{equation*}
\left(
\begin{array}{cc}
0 & u^t\\
u & A
\end{array}
\right),
\end{equation*}
where $u$ is a column vector in $\R^{2n}$ and $A \in \so(2n)$. Those elements with $u=0$ generate $\so(2n) \subset \so(2n,1)$. The Killing form is definite on $\so(2n,1)$ and so gives an equivariant isomorphism $\so(2n,1) \cong \so(2n,1)^*$. As in \S\ref{coadjoint description}, we consider the orbit of
$$
\xi=\left(
\begin{array}{cc}
0 & 0\\
0 & J_0
\end{array}
\right)
$$
where $J_0\in \so(2n)$ is a choice of almost complex structure on $\R^{2n}$ (i.e., $J_0^2=-1$). The stabiliser of $\xi$ is $\U(n)$. We write $Z_{2n}$ for the coadjoint orbit of $\xi$. 

Note that $Z_{2n} \cong \SO(2n,1)/\U(2n)$ is symplectic of dimension $n(n+1)$. It fibres over hyperbolic space $H^{2n} \cong \SO(2n,1)/\SO(2n)$ with fibre isomorphic to the space $\SO(2n)/\U(n)$ of orthogonal complex structures on $\R^{2n}$ inducing a fixed orientation. In other words, it is the twistor space of~$H^{2n}$.

The same proof as in Lemma \ref{U(2) decomposition} gives the following result.

\begin{lemma}\label{U(n) decomposition}
There is an isomorphism of $U(n)$-representation spaces:
$$
\so(2n,1) \cong \u(n) \oplus \Lambda^2(\C^n)^* \oplus \C^n.
$$
Given a point $z\in Z_{2n}$ with stabiliser $\U(n) \subset 
\SO(2n,1)$ there is a $U(n)$-equivariant isomorphism 
\begin{equation}
T_z \cong \Lambda^2(\C^n)^* \oplus \C^n,
\label{tangent splitting}
\end{equation}
in which the $\Lambda^2(\C^n)^*$ summand is tangent to the fibre of the projection $Z_{2n} \to H^{2n}$.
\end{lemma}

As in Lemma \ref{uniqueness of structure}, by $\U(n)$-equivariance, the symplectic form on $T_z$ is proportional under (\ref{tangent splitting}) to the form induced by the Euclidean structure on $\C^n$. To show this constant of proportionality is positive, first check that the forms are genuinely equal in the case $n=1$, where $\so(2,1) \cong \u(1) \oplus\C$ as a $\U(1)$-representation. This amounts to the fact that $Z_2 = H^2$ with symplectic form the hyperbolic area form. Next, use induction and the fact that the decompositions of $\so(2n,1)$ and $\so(2n+2,1)$ from Lemma \ref{U(n) decomposition} are compatible with the obvious inclusions of the summands induced by a choice of $\C^n \subset \C^{n+1}$.

Having seen that the isomorphism of (\ref{tangent splitting}) is symplectic, we define a compatible $\SO(2n,1)$-invariant almost complex structure on $Z_{2n}$ by declaring (\ref{tangent splitting}) to be a complex linear isomorphism. With this almost complex structure, $TZ = V \oplus H$ splits as a sum of complex bundles with $V \cong \Lambda^2H^*$, corresponding to (\ref{tangent splitting}). 

In the twistorial picture, this is simply the decomposition of $TZ$ induced by the Levi--Civita connection of $H^{2n}$. The almost complex structure here is the (non-integrable) ``Eells--Salamon'' structure, given by reversing the (integrable) ``Atiyah--Hitchin--Singer'' structure in the vertical directions.

\begin{lemma}
$c_1(Z_{2n}) = (2-n)c_1(H)$.
\end{lemma}
\begin{proof}
This follows from $TZ = \Lambda^2H^*\oplus H$ along with the fact that for any complex rank $n$ vector bundle $E$, $c_1(\Lambda^2E)=(n-1)c_1(E)$. 
\end{proof}

We now determine the symplectic class of $Z_{2n}$. First, consider the restriction of the symplectic structure to the fibres of $Z_{2n} \to H^{2n}$. It follows from (\ref{tangent splitting}) that the fibres are symplectic and almost-complex submanifolds. Moreover, the stabiliser $\SO(2n) \subset \SO(2n,1)$ of a point $x \in H^{2n}$ acts on the fibre $F_x$ over $x$ preserving both these structures. As mentioned above, $F_x \cong \SO(2n)/\U(n)$ is the space of orthogonal complex structures on $T_xH^{2n}$ inducing a fixed orientation. The standard theory of symmetric spaces gives $F_x$ a symmetric Kähler structure. It follows from $\SO(2n)$-equivariance that this must agree with the restriction of the symplectic and almost complex structures from~$Z_{2n}$.

It is also standard that the symmetric Kähler structure on $F = \SO(2n)/\U(n)$ comes from a projective embedding. Let $E \to F$ denote the ``tautological'' bundle: each point of $F$ is a complex structure on $\R^{2n}$; the fibre of $E$ at a point $J \in F$ is the complex vector space $(\R^{2n}, J)$. The bundle $\det E^*$ is ample and $c_1(\det E^*) = -c_1(E)$ is represented by the symmetric symplectic form on $F$. 

In our situation, the splitting (\ref{tangent splitting}) tells us that the tautological bundle of the fibre $F_x$ is simply $H|_{F_x}$. It follows that on restriction to a fibre, the symplectic class agrees with $-c_1(H)$. However, topologically, $Z_{2n} \cong F \times H^{2n}$ is homotopic to $F$. Since their fibrewise restrictions agree, it follows that $-c_1(H)$ is equal to the symplectic class of $Z_{2n}$. Hence, writing $\omega$ for the symplectic form on $Z_{2n}$, we have:

\begin{proposition}\label{first chern class}
$c_1(Z_{2n}) = (n-2) [\omega]$.
\end{proposition}

\subsection{Compact quotients}

The symplectic form and almost complex structure on $Z_{2n}$ as well as the splitting $TZ_{2n} = V \oplus H$ and identification $V \cong \Lambda^2 H^*$ are $\SO(2n,1)$-invariant. It follows that all these arguments, in particular Proposition \ref{first chern class}, apply equally to smooth quotients of $Z_{2n}$ by subgroups of $\SO(2n,1)$. This means we can find compact examples of symplectic manifolds for which $c_1$ is a positive multiply of $[\omega]$. Let $\Gamma \subset \SO(2n,1)$ be the fundamental group of a compact hyperbolic $2n$-manifold $M$. $\Gamma$ acts by symplectomorphisms on $Z_{2n}$ to give as quotient a symplectic manifold of dimension $n(n+1)$. It fibres $Z_{2n}/\Gamma \to H^{2n}/\Gamma$ over $M$ as the twistor space of $M$.  When $n\geq 3$, these are compact symplectic manifolds for which $c_1$ is a positive multiple of $[\omega]$ (for $n\geq 3$). Such manifolds can never be Kähler since they have hyperbolic fundamental group. Indeed, Kähler Fano manifolds are even simply connected. These examples (originally appearing in Reznikov's article \cite{reznikov}) are, to the best of our knowledge, the first non-Kähler symplectic ``Fano'' manifolds. The lowest dimension which can be achieved in this way is 12. In this case, the fibration over $M^6$ has fibres $\C\P^3 \cong \SO(6)/\U(3)$.

\section{Concluding remarks} \label{concluding remarks}

The constructions presented here seem, to us at least, to lead several natural questions. We describe some of these below.
 
\subsection{Simply connected symplectic examples}

Whilst we only give one example of a simply-connected symplectic manifold with $c_1=0$, we plan to exploit a similar construction to produce simply connected 6-dimensional examples with arbitrarily large Betti numbers (see the forthcoming article \cite{fine-panov3}). By comparison, note that it is still unknown if there are infinitely many topologically distinct \emph{Kähler} Calabi--Yau manifolds of fixed dimension.

First we describe an infinite sequence of simply-connected compact symplectic orbifolds with $c_1=0$. They come from hyperbolic orbifold metrics on $S^4$, built using Coxeter polytopes. Recall that a \emph{hyperbolic Coxeter polytope} is a convex polytope with totally geodesic boundary in $H^n$ and whose dihedral angles are $\pi/k$ for $k\in \N$; the polytope is said to be \emph{right-angled} if all dihedral angles are $\pi/2$.

\begin{lemma}
Let $P$ be an $n$-dimensional compact hyperbolic Coxeter polytope. Doubling $P$ gives a hyperbolic orbifold metric on $S^n$.
\end{lemma}
\begin{proof}
Let $G$ be the group of isometries of $H^n$ generated by reflections in the faces of $P$ and let $G'$ be the subgroup of $G$ of index two consisting of all orientation preserving elements. Then $H^n/G'$ is the double of $P$. Indeed the fundamental domain of $G$ is $P$ itself, whilst the fundamental domain of $G'$ is $P \cup P';$ where $P'$ is the reflection of $P$ in a face. Identifying further the faces of $P\cup P'$ via $G'$ we obtain the double of~$P$. The double is homeomorphic to $S^n$, since $P$ is homeomorphic to a closed $n$-ball.
\end{proof}

Up to dimension 6, compact hyperbolic Coxeter polytopes are known to be abundant.

\begin{theorem}[Potyagilo--Vinberg \cite{potyagilo-vinberg}, Allcock \cite{allcock}]
\label{hyperbolic Coxeter polytopes}
In all dimensions up to 6 there exist infinitely many compact hyperbolic Coxeter polytopes. Moreover, in dimensions up to 4 there are infinitely many compact hyperbolic right-angled polytopes. 
\end{theorem}

We are interested in the case of 4 and 6 dimensions. The right angled 4-dimensional polytopes are constructed by Potyagilo--Vinberg \cite{potyagilo-vinberg}. Here one uses that there is a hyperbolic 120-cell with dihedral angles $\pi/2$. Gluing it to itself along a 3-face (via the reflection in that face) produces a new right-angled polytope. This procedure can be repeated giving infinitely many examples. An infinite family of 6-dimensional polytopes is constructed by Allcock \cite{allcock}. 

Doubling the 4-dimensional hyperbolic Coxeter polytopes gives infinitely many simply-connected symplectic 6-orbifolds with $c_1=0$. 

When the polytope is doubled, the hyperbolic metric extends smoothly across the 3-faces; the singularities correspond to the 2-skeleton of the polytope. To understand the singularities in the twistor space, consider the positive ``octant'' $x_1,x_2,x_3,x_4 \geq 0$ in $\R^4$, i.e., the infinitesimal model for the vertex of a right-angled Coxeter polytope. The $(x_1,x_2)$-plane lifts to two planes $L, L'$ in the twistor space, corresponding to the two compatible complex structures on $\R^4$ for which the $(x_1,x_2)$-plane is a complex line. Reflection in the $(x_1,x_2)$-plane lifts to the twistor space where it fixes $L, L'$ pointwise. The other coordinate 2-planes behave similarly. The lifts of coordinate 2-planes are some of the points with non-trivial stabiliser under the action on the twistor space generated by the reflections. The only other points with non-trivial stabiliser are those on the twistor line over the origin. This line is fixed pointwise by the composition of reflection in one 2-plane with reflection in the orthogonal 2-plane.

It follows that the orbifold singularities of the twistor space of a doubled right-angled Coxeter polytope are of two sorts. The generic orbifold point has structure group $\Z_2$ and is modelled on the quotient of $\C^3$ by $(z_1,z_2,z_3) \mapsto (z_1, -z_2,-z_3)$. These correspond either to the lifted coordinate 2-planes or the central twistor line in the above picture. Then there are isolated points in the singular locus, where three surfaces of generic orbifold points meet. Here the structure group is $(\Z_2)^2$. This is most symmetrically described via the action of $(\Z_2)^3$ on $\C^3$ where each generator changes the sign on one of the three coordinate 2-planes in $\C^3$; the diagonal $\Z_2$ acts trivially so the action factors through $(\Z_2)^2$. Such points correspond in the picture above to the intersection of the twistor line at the origin with the lifts of two orthogonal coordinate 2-planes.

The concrete description of these singularities means that it is possible to find  symplectic crepant resolutions ``by hand'', as it was for the Davis orbifold. This gives an infinite collection of simply-connected symplectic manifolds with $c_1=0$ \cite{fine-panov3}. Meanwhile, the twistor spaces of the 6-orbifolds of Allcock give an infinite collection of simply-connected symplectic Fano 12-orbifolds. It is natural to ask if these admit symplectic Fano resolutions, although this looks much harder to answer than the corresponding question for the 6-orbifolds.

On the subject of symplectic resolutions, we mention the recent work of Nieder\-krüger--Pasquotto \cite{niederkruger-pasquotto,niederkruger-pasquotto2}, which gives a systematic approach to the resolution of symplectic orbifolds arising via symplectic reduction, although they do not consider discrepancy.

\subsection{Possible diversity of symplectic 6-manifolds with $c_1=0$}

On the subject of hyperbolic orbifolds, there is a much more general existence question, which we learnt from Gromov:

\begin{question}
Is there any restriction on the manifolds which can be obtained as quotients of $H^n$ by a cocompact discrete subgroup of $\Orth(n,1)$?
\end{question}

(The subgroup is allowed to have torsion, of course.) For $n=2$ and $3$ all compact manifolds can be obtained as quotients ($n=2$ is straightforward whilst $n=3$ uses geometrization). The relationship between hyperbolic and symplectic geometry outlined here gives additional motivation to try to answer the question in dimension four. For example, can any finitely presented group be the fundamental group of such a quotient? The 4-dimensional quotients give rise to 6-dimensional symplectic orbifolds with $c_1=0$ and in this way one might ambitiously hope to approach the problem of which groups can appear as the fundamental group of a symplectic manifold with $c_1=0$. Of course, even if one had the orbifolds, they would still need to be resolved. In algebraic geometry, crepant resolutions of threefolds always exist, thanks to the work of Bridgeland--King--Reid \cite{BKR}. Even independently of the hyperbolic orbifold approach described here, it would be interesting to know what holds in the symplectic setting. 

\subsection{Rational curves in hyperbolic twistor spaces}

We conclude our discussion of the symplectic examples with a brief look at the genus zero Gromov--Witten invariants of hyperbolic twistor spaces.

\begin{lemma}
Let $\C\P^1 \to Z_{2n}$ be pseudoholomorphic with respect to the $\SO(2n,1)$-invariant almost complex structure defined above. Then the image lies entirely in a fibre of $Z_{2n} \to H^{2n}$.
\end{lemma}

\begin{proof}
A theorem of Salamon \cite{salamon} (Eells--Salamon \cite{eells-salamon} in the case $n=2$) shows that the image in $H^{2n}$ of a pseudoholomorphic curve in $Z_{2n}$ is a minimal surface. Since $H^{2n}$ contains no minimal spheres the result follows.
\end{proof}

It follows that there are pairs of points in $Z_{2n}$ (or any of its compact quotients) which do not lie on a pseudoholomorphic rational curve. This is in stark contrast to the situation for Kähler Fano manifolds. Of course, we consider here a \emph{particular} almost complex structure. The symplectic definition of rational connectivity (see Li--Ruan \cite{li-ruan}) involves the non-vanishing of certain genus zero Gromov--Witten invariants. A weaker property than rational connectivity is that of being uniruled. In algebraic geometry a variety is called uniruled if each point is contained in a rational curve. Again, the definition in symplectic geometry involves a statement about genus zero Gromov--Witten invariants. Li--Ruan \cite{li-ruan} have asked if all symplectic Fano manifolds are uniruled in this sense.

We content ourselves here by mentioning that the above restriction on the image of rational curves suggests an approach to computing the genus zero Gromov--Witten invariants of the twistor spaces of hyperbolic $2n$-mani\-folds. Whilst the invariant almost complex structure is not generic, the obstruction bundle should be describable in simple terms. Indeed, using this lemma it should be possible to localise calculations to the case of a curve in the fibre of $Z_{2n}\to H^{2n}$ and exploit the action of the $\SO(2n)$ of isometries fixing the point in $H^{2n}$.

\subsection{Miles Reid's fantasy}

In \cite{reid} Miles Reid asked a question, which is now referred to as \emph{Reid's fantasy}:

\begin{question}[Reid \cite{reid}] 
Consider the moduli space of complex structures on $(S^3\times S^3)^{\#N}$ with $K\cong \mathcal O$ which are deformations of Moishezon spaces. Is it true that for large $N$ this space is irreducible?

Do all simply-connected Kähler Calabi-Yau threefolds appear as small resolutions of 3-folds with double points lying on the boundary of these moduli spaces?
\end{question}

The complex threefolds that we obtain do not appear directly in Reid's fantasy and indeed seem to be of a very different nature. There is no visible mechanism that would enable one to connect these examples in any way to Moishezon manifolds.  It seems more reasonable that the structures we construct on $2(S^3\times S^3)\#(S^2\times S^4)$ belong to an infinite family of disconnected components of complex structures with $K\cong \mathcal O$. If this were true, it would show that Reid was wise to limit himself in his fantasies when he wrote ``I aim to consider only analytic threefolds which are deformations of Moishezon spaces'' (\cite{reid}, page 331). Of course we don't know how one could try to prove (or disprove) the existence of this infinite number of connected components.

\bibliographystyle{plain}
\bibliography{hyp_geom_CY_bib}

{\small \noindent {\tt joel.fine@ulb.ac.be }} \newline
{\small D\'epartment de Math\'ematique,
Universit\'e Libre de Bruxelles CP218,
Boulevard du Triomphe,
Bruxelles 1050,
Belgique.}

{\small \noindent {\tt d.panov@imperial.ac.uk }} \newline
{\small Department of Mathematics,
Imperial College London,
South Kensington Campus, 
London SW7 2AZ, United Kingdom.}

\end{document}